\documentclass[11pt]{amsart}
\pdfoutput=1

\usepackage{amsmath, amssymb, amsthm, mathtools, enumitem}
\usepackage{accents, xcolor}
\usepackage{cancel}
\usepackage{amssymb}
\usepackage{amsthm}
\usepackage{amsfonts}
\usepackage{amsmath}
\usepackage{pmboxdraw}
\usepackage{verbatim} 
\usepackage{graphicx}
\usepackage{color}
\usepackage[colorlinks=true, citecolor=blue, filecolor=black, linkcolor=black, urlcolor=black]{hyperref}
\usepackage{cite}
\usepackage[normalem]{ulem}
\usepackage{subcaption}
\usepackage{todonotes}
\usepackage{kantlipsum}
\allowdisplaybreaks
\usepackage{bbm}
\usepackage{tikz}

\newtheorem{theorem}{Theorem}[section]
\newtheorem{lemma}[theorem]{Lemma}

\newtheorem{corollary}[theorem]{Corollary}

\theoremstyle{definition}
\newtheorem{rem}[theorem]{Remark}

\newcommand{\erf}{\operatorname{erf}}

\newcommand{\gaussF}{{}_{2} \mathbf{F}_1}

\newcommand{\Airy}{\operatorname{Ai}}

\newcommand{\re}{\operatorname{Re}}
\newcommand{\im}{\operatorname{Im}}

\def\R{\mathbb{R}}
\newcommand{\bfR}{\mathbf{R}}

\numberwithin{equation}{section}

\author{Sung-Soo Byun}
\address{Department of Mathematical Sciences and Research Institute of Mathematics, Seoul National University, Seoul 151-747, Republic of Korea}
\email{sungsoobyun@snu.ac.kr}

\author{Yong-Woo Lee}
\address{Department of Mathematical Sciences, Seoul National University, Seoul 151-747, Republic of Korea}
\email{hellowoo@snu.ac.kr}

\begin{document}
\title[Finite size corrections of the elliptic G{\SMALL in}OE]{Finite size corrections for real eigenvalues \\
of the elliptic Ginibre matrices }



\thanks{Sung-Soo Byun was partially supported by the POSCO TJ Park Foundation (POSCO Science Fellowship) and by the New Faculty Startup Fund at Seoul National University. Yong-Woo Lee was partially supported by Samsung Science and Technology Foundation (SSTF-BA1401-51) and by the KIAS Student Fellow at Korea Institute for Advanced Study.
}

\begin{abstract}
We consider the elliptic Ginibre matrices in the orthogonal symmetry class that interpolates between the real Ginibre ensemble and the Gaussian orthogonal ensemble. We obtain the finite size corrections of the real eigenvalue densities in both the global and edge scaling regimes, as well as in both the strong and weak non-Hermiticity regimes. 
Our results extend and provide the rate of convergence to the previous recent findings in the aforementioned limits. In particular, in the Hermitian limit, our results recover the finite size corrections of the Gaussian orthogonal ensemble established by Forrester, Frankel and Garoni.
\end{abstract}

\maketitle

\section{Introduction and main results}

By their very nature, elliptic random matrices were introduced to interpolate between Hermitian and non-Hermitian random matrix theories, see \cite[Section 2.3]{BF22} and \cite[Sections 2.8 and 5.5]{BF23} for recent reviews.
One of the simplest ways to define the model in the orthogonal symmetry class is by starting with the real Ginibre matrix (denoted GinOE) $X$, an $N \times N$ matrix whose elements are given by independent real Gaussian random variables with mean $0$ and variance $1/N$.
As expected from its fundamental structure, the Ginibre matrix serves as a cornerstone in non-Hermitian random matrix theory. 
To introduce the elliptic random matrix model, one usually makes use of a parameter $\tau \equiv \tau_N \in [0,1]$. 
This parameter determines the Hermiticity of the model and may depend on the matrix dimension $N$. 
Then the one-parameter generalisation of the GinOE, known as the elliptic GinOE, is defined by
\begin{equation}
X^{(\tau)}:= \sqrt{ \frac{1+\tau}{2} } S_+ + \sqrt{ \frac{1-\tau}{2}  } S_-,  \qquad S_\pm = \frac{ X+X^T }{ \sqrt{2} }. 
\end{equation}
Notice here that for $\tau=0,$ the matrix $X^{(\tau)}$ recovers the GinOE matrix, whereas for $\tau=1$, it recovers the Gaussian orthogonal ensemble (GOE), cf. \cite[Chapter 1]{Fo10}.  
Let us mention that the terminology \emph{elliptic} stems from the fact that the eigenvalues of $X^{(\tau)}$ tend to uniformly occupy the ellipse 
\begin{equation} \label{ellipse}
\Big\{ (x,y) \in \R^2 : \Big( \frac{x}{1+\tau} \Big)^2 + \Big( \frac{y}{1-\tau} \Big)^2 \le 1   \Big\}, 
\end{equation}
which is known as the elliptic law, see e.g. \cite{NO15,AK22,By23a}. 
We refer to \cite{ADM22, BE22, Mo22, OYZ23,OR16,LR16,FG23,FT21,Fo23,HHJK23,By23b} and references therein for recent work on the elliptic Ginibre ensembles.

Compared to their complex or quaternion counterparts, a characteristic feature of real random matrices is their non-trivial probability of having real eigenvalues. This arises because the characteristic polynomial of the real matrix model consists of real coefficients, resulting in eigenvalues that are either purely real or form complex conjugate pairs, see e.g. \cite{NV21, AK07, AB22, FS23a, KPTTZ15, Si17, LMS22, FM12, Si17a,WCF23,GLX23,CESX22} for related models. 
From a statistical physics point of view, this behaviour resembles that of two-species particle systems, such as a two-component plasma in the plane.
In particular, the statistics of real eigenvalues enjoy surprising connections to other fields, such as annihilating Brownian motions \cite{TZ11,GPTZ18} and diffusion processes \cite{Fo15a,SM07}, and also find applications for instance in the context of the equilibrium counting \cite{FK16}, so-called mermaid states in certain topologically protected quantum dots \cite{BEDPMW13}. 

\medskip 

In this work, we shall study real eigenvalues of the elliptic GinOE focusing on finite size effects in scaling limits. 
This in turn relates to the question of the speed of convergence to universal laws exhibited by the elliptic GinOE real eigenvalues.
For this purpose, we denote by $\mathcal{N}_\tau$ the (random) number of real eigenvalues. 
The basic statistical information of real eigenvalues is encoded in the $1$-point function $\bfR_N: \R \to \R_+$, which is defined by its characterizing property: for a test function $f:\R \to \R$, 
\begin{equation} \label{EN RN int}
\mathbb{E} \bigg[\sum_{ j=1 }^{ \mathcal{N}_{\tau}  }  f( x_j ) \bigg] = \int_{ \mathbb{R} } f(x) \, \textbf{R}_N(x)\,dx .
\end{equation}
In particular, by letting $f \equiv 1$, it gives rise to the expected number of real eigenvalues 
\begin{equation}
E_{N,\tau} := \mathbb{E} \mathcal{N}_\tau = \int_\R \bfR_N(x)\,dx. 
\end{equation}
Then the normalised density $\rho_N \equiv \rho_{N,\tau}$ of real eigenvalues is defined by 
\begin{equation} \label{rho N RN}
\rho_N(x):= \frac{1}{ E_{N,\tau} } \bfR_N(x). 
\end{equation}

In the study of the large-$N$ asymptotics of the elliptic Ginibre matrices, the following two distinct regimes emerge.
\begin{itemize}
    \item \textbf{Strong non-Hermiticity}. Here, $\tau$ is fixed in the interval $[0,1)$. In particular if $\tau=0$, it coincides with the GinOE. 
    \smallskip
    \item \textbf{Weak non-Hermiticity}. Here, $\tau \uparrow 1$ with a proper speed. In this regime, one again needs to distinguish two different scales: for a fixed parameter $\alpha \in [0,\infty),$ \smallskip 
    \begin{itemize}
        \item \emph{Bulk scaling}: 
        \begin{equation} \label{tau wH bulk}
        \tau \equiv \tau_N= 1-\frac{\alpha^2}{N},
        \end{equation}
        \item \emph{Edge scaling}: 
        \begin{equation} \label{tau wH edge}
        \tau \equiv \tau_N = 1- \frac{\alpha^2}{N^{1/3}}. 
        \end{equation}
    In particular, if $\tau=1$ (i.e. $\alpha=0$), it corresponds to the GOE. 
    \end{itemize}
\end{itemize}
The weakly non-Hermitian regime was introduced by Fyodorov, Khoruzhenko and Sommers \cite{FKS97, FKS98,FKS97a}, where they studied the critical transition of the elliptic GinUE statistics in the bulk scaling \eqref{tau wH bulk}.
On the one hand, the edge scaling \eqref{tau wH edge} has been employed in \cite{Be10, AP14, BES23} to derive non-Hermitian extensions of the Airy point processes.
We also refer to \cite{ACV18, AB23, BL22a, BMS23} and references therein for further recent work on the elliptic Ginibre ensembles at weak non-Hermiticity.

\medskip 
Turning back to the real eigenvalue statistics, the recent work \cite{BKLL23} obtained the full asymptotic expansion of the expected number $E_{N,\tau}$ of real eigenvalues.
From now on, we shall focus on the case where the matrix dimension $N$ is even. The odd $N$ case should be treated separately, as the analysis is usually more involved, see e.g. \cite{Si07,FM09}.

In the strongly non-Hermitian regime where $\tau \in [0,1)$ is fixed, it was established in \cite[Proposition 2.1]{BKLL23} that 
\begin{equation} \label{EN tau sH}
    E_{N,\tau} = \Big(\frac{2}{\pi}\frac{1+\tau}{1-\tau}N\Big)^\frac{1}{2} + \frac{1}{2} + O(N^{-\frac{1}{2}}),
\end{equation}
as $N \to \infty$. 
The leading order asymptotic behaviour of \eqref{EN tau sH} for the GinOE case ($\tau=0$) was first obtained in the celebrated work \cite{EKS94} of Edelman, Kostlan and Shub. 
For general $\tau \in [0,1)$ fixed, this leading order behaviour was then later derived by Forrester and Nagao \cite{FN08}.
From the formula \eqref{EN tau sH}, one can observe that $E_{N,\tau}$ is an increasing function of $\tau$.
This behaviour is intuitively clear, as increasing $\tau$ brings the model closer to a symmetric matrix.
Let us also mention that the behaviour \eqref{EN tau sH} indeed holds as long as $1-\tau \gg 1/N$. 
We stress that while the leading order asymptotic $O(\sqrt{N})$ term in \eqref{EN tau sH} depends on $\tau$, the subleading order $1/2$ does not depend on $\tau$. 
This fact is closely related to the universal edge scaling limit, cf. Remark~\ref{Rem_edge universality}.   

The asymptotic behaviour of $E_{N,\tau}$ in the bulk scaling weak non-Hermiticity regime \eqref{tau wH bulk} exhibits different and seemingly more complicated behaviour. 
In this case, it was shown in \cite[Theorem 2.1]{BKLL23} that 
\begin{equation}  \label{EN tau wH}
   E_{N,\tau} = Nc(\alpha) + c_0(\alpha) + \frac{1}{2} + O(N^{-1}),
\end{equation}
as $N \to \infty$, where 
\begin{align}
\label{c(alpha)}
c(\alpha)&:= e^{-\alpha^2/2} [ I_0( \tfrac{\alpha^2}{2} )+ I_1( \tfrac{\alpha^2}{2} ) ],
\\ \label{c0(alpha)}
c_0(\alpha) &:=-\tfrac12 e^{-\alpha^2/2} [ I_0( \tfrac{\alpha^2}{2} )+\alpha^2 I_1( \tfrac{\alpha^2}{2} ) ]. 
\end{align}
Here, $I_\nu$ is the modified Bessel function of the first kind,
see e.g. \cite[Chapter 10]{NIST}.

\medskip 

Unlike the precise asymptotic expansion of the expected number of real eigenvalues, for general $\tau \in [0,1]$, the asymptotic behaviour of the 1-point function $\bfR_N$ is  available in the literature only for the leading order.
(See however Remarks~\ref{Rem_Global GinOE} and ~\ref{Rem_GOE density} for the GinOE case $\tau=0$ and the GOE case $\tau=1$ respectively.)
Nevertheless, the precise asymptotic behaviour, often referred to as the \emph{finite size correction} in random matrix theory, encompasses important statistical properties such as the counting statistics \cite{ACCL22, Ch22, Ch23, ACM23, ABES23,DDMS19,SLMS22,SLMS21}, which also carries physical implications.
On the one hand, for the Hermitian case, there have been extensive works on the finite size corrections of eigenvalue distributions, see e.g. \cite{PS16,Bo16,FFG06,GFF05,FT19,YZ23,FLT21,FS23} and references therein. 
Furthermore, these have applications in various contexts, including the statistics of the critical zeros of the Riemann zeta function \cite{FM15} and the distribution of the longest increasing subsequence \cite{FM23,Bo23}.

\medskip 

In our first result, we obtain the finite size correction of the global eigenvalue densities for both strong and weak non-Hermiticities. 

\begin{theorem}[\textbf{Finite size correction of the $1$-point functions}] \label{Thm_Bulk density} 
Let $N$ be an even integer. Then as $N \to \infty$, we have the following.
\begin{itemize}
    \item[\textup{(i)}] \textup{\textbf{(Strong non-Hermiticity)}} Let $\tau \in [0,1)$ be fixed. Then for any $x \in (-1-\tau, 1+\tau)$, we have 
    \begin{equation} \label{RN asymp sH}
    \bfR_N(x) = \bfR_{(0)}^{ \rm s }(x) N^{\frac12}  + O(e^{-\epsilon N}), \qquad \bfR_{(0)}^{ \rm s }(x):= \Big(\frac{1}{2\pi(1-\tau^2)}\Big)^{\frac12}
    \end{equation}
    for some $\epsilon>0.$
    \smallskip 
    \item[\textup{(ii)}] \textup{\textbf{(Weak non-Hermiticity)}}  Let $\tau = 1 - \alpha^2/N$ with fixed $\alpha \in [0,\infty)$. 
    Then for any $x \in (-2,2)$, we have 
    \begin{equation}  \label{RN asymp wH}
    \bfR_N(x) =  \bfR_{(0)}^{ \rm w } (x) \, N + \bfR_{(1)}^{ \rm w }(x) + O(N^{-1}),
    \end{equation}
    where
    \begin{align}
    \bfR_{(0)}^{ \rm w }(x) &:= \frac{1}{2\alpha\sqrt{\pi}} \erf(\tfrac{\alpha}{2}\sqrt{4-x^2}),  
    \\
    \bfR_{(1)}^{ \rm w }(x) &:= \frac{\alpha}{8\sqrt{\pi}} \erf(\tfrac{\alpha}{2}\sqrt{4-x^2})
    - \frac{3\alpha^2x^2+4-4\alpha^2}{8\pi\sqrt{4-x^2}} e^{\frac{\alpha^2}{4}(x^2-4)}. 
    \end{align}
\end{itemize}
\end{theorem}

As an immediate consequence of Theorem~\ref{Thm_Bulk density} together with \eqref{rho N RN}, \eqref{EN tau sH} and \eqref{EN tau wH}, we have the following corollary. 

\begin{corollary}[\textbf{Finite size correction of the normalised densities}] \label{Cor_normalised global}
Let $N$ be an even integer. Then as $N \to \infty$, we have the following.
\begin{itemize}
    \item[\textup{(i)}] \textup{\textbf{(Strong non-Hermiticity)}} Let $\tau \in [0,1)$ be fixed. Then for any $x \in (-1-\tau, 1+\tau)$, we have 
    \begin{equation}  \label{rhoN asymp sH}
    \rho_N(x) = \rho_{(0)}^{ \rm s }(x) +  \rho_{(1)}^{ \rm s } (x) \, N^{-\frac{1}{2}} + O(N^{-1}),
    \end{equation}
    where 
    \begin{equation} \label{rho 0 1 sH}
    \rho_{(0)}^{ \rm s }(x):= \frac{1}{2(1+\tau)}, \qquad  \rho_{(1)}^{ \rm s }(x):=  -  \Big( \frac{\pi(1-\tau)}{32(1+\tau)^3} \Big)^\frac{1}{2}.
    \end{equation}
    
    \item[\textup{(ii)}] \textup{\textbf{(Weak non-Hermiticity)}}  Let $\tau = 1 - \alpha^2/N$ with fixed $\alpha \in [0,\infty)$. 
    Then for any $x \in (-2,2)$, we have 
    \begin{equation}  \label{rhoN asymp wH}
    \rho_N(x) = \rho_{(0)}^{ \rm w }(x) +  \rho_{(1)}^{ \rm w } (x) \, N^{-1} + O(N^{-2}),
    \end{equation}
    where
    \begin{align}
    \rho_{(0)}^{ \rm w }(x) &:= \frac{1}{c(\alpha)}\frac{1}{2\alpha\sqrt{\pi}} \erf(\tfrac{\alpha}{2}\sqrt{4-x^2}),
    \\
    \rho_{(1)}^{ \rm w }(x) &:= \frac{1}{c(\alpha)^2}\frac{c(\alpha)\alpha^2 - 4 c_0(\alpha)-2}{8\alpha\sqrt{\pi}} \erf(\tfrac{\alpha}{2}\sqrt{4-x^2})
    - \frac{1}{c(\alpha)}\frac{3\alpha^2x^2+4-4\alpha^2}{8\pi\sqrt{4-x^2}} e^{\frac{\alpha^2}{4}(x^2-4)}.
    \end{align}
    Here $c(\alpha)$ and $c_0(\alpha)$ are given by \eqref{c(alpha)} and \eqref{c0(alpha)}.
\end{itemize}
\end{corollary}

See Figure~\ref{Fig_Global density} for the numerics on Theorem~\ref{Thm_Bulk density}. 

\begin{figure}[h!]
	\begin{subfigure}{0.45\textwidth}
		\begin{center}
			\includegraphics[width=\textwidth]{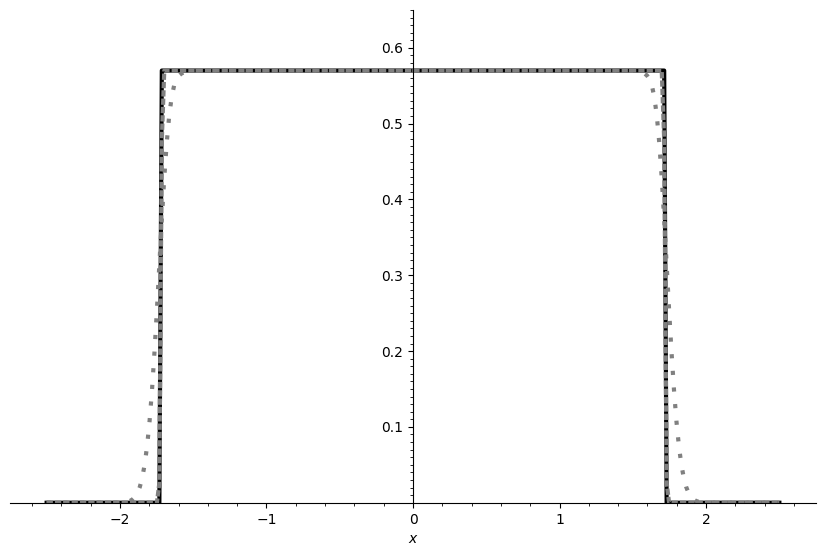}
		\end{center}
		\subcaption{ $N^{-1/2} \bfR_N$ and $ \bfR_{(0)}^{ \rm s}$ }
	\end{subfigure}
   	\begin{subfigure}{0.45\textwidth}
		\begin{center}
			\includegraphics[width=\textwidth]{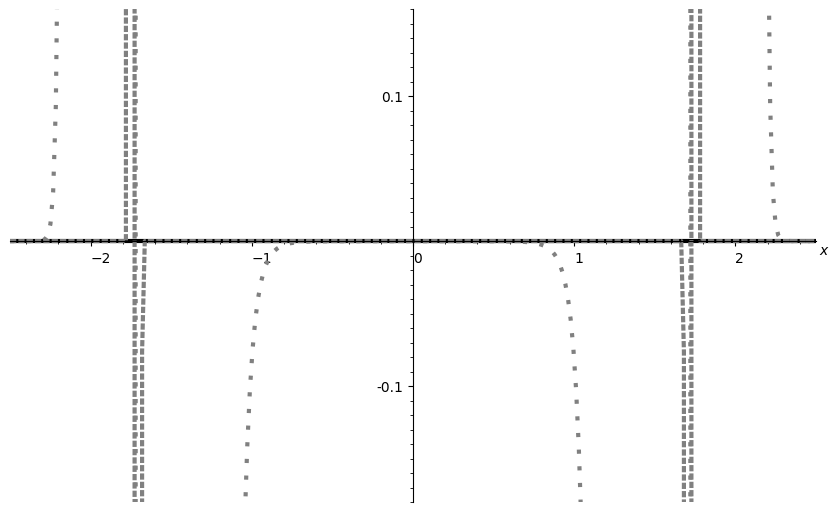}
		\end{center}
		\subcaption{  $N^{7/2} (\bfR_N-N^{1/2}\bfR_{(0)}^{ \rm s}) $  }
	\end{subfigure}

         \begin{subfigure}{0.45\textwidth}
		\begin{center}
			\includegraphics[width=\textwidth]{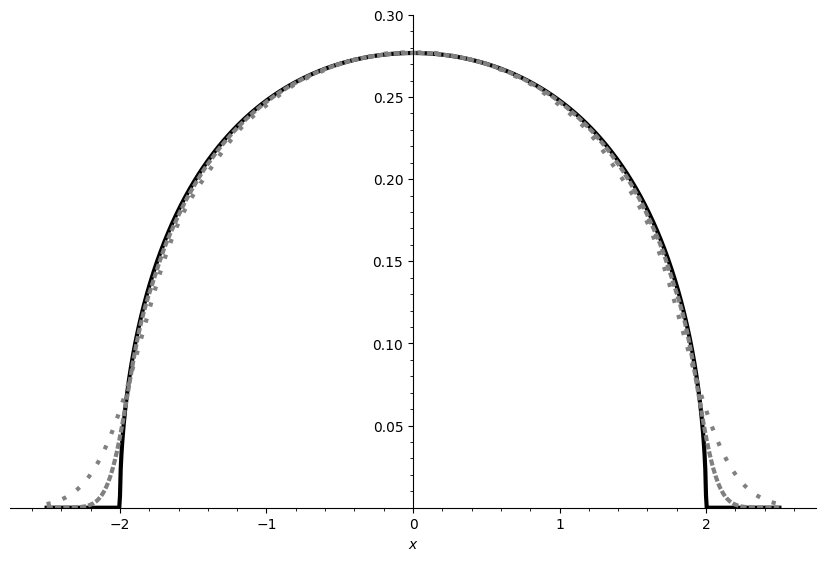}
		\end{center}
		\subcaption{ $N^{-1} \bfR_N$ and $\bfR_{(0)}^{ \rm w} $ }
	\end{subfigure}
   	\begin{subfigure}{0.45\textwidth}
		\begin{center}
			\includegraphics[width=\textwidth]{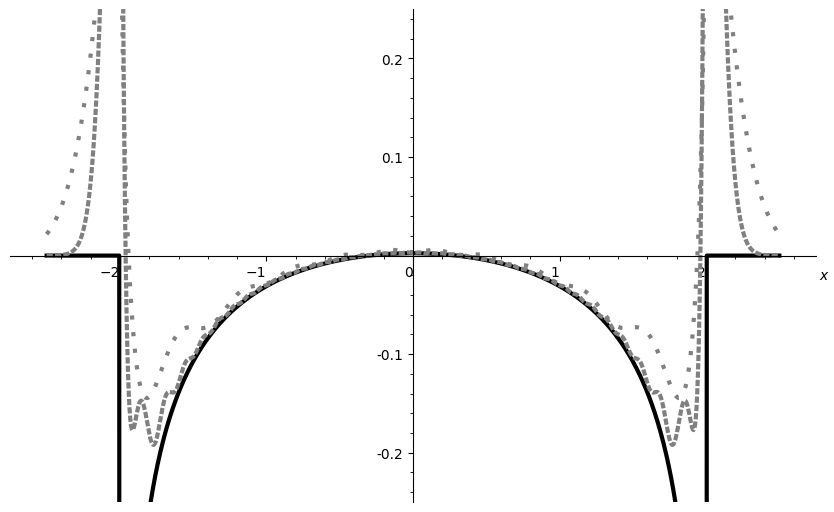}
		\end{center}
		\subcaption{  $\bfR_N-N\,\bfR_{(0)}^{ \rm w}$ and $\bfR_{(1)}^{ \rm w}$  }
	\end{subfigure}
	\caption{ The plot (A) shows the function $N^{-1/2} \bfR_N$ converging to the limit $ \bfR_{(0)}^{ \rm s } $ in \eqref{RN asymp sH} for $N = 80,\ 5120$ with $\tau = 5/7$. The plot (B) is the graph of $N^{7/2} (\bfR_N - N^{1/2} \bfR_{(0)}^{\rm s})$ with the same values of $N$ and $\tau$. 
Here, the exponent $7/2$ is arbitrarily chosen to show the exponential decay.  
 The plots (C) and (D) are analogous figures for the weak non-Hermiticity regime with $N = 10,\ 40$ and $\alpha = 2/3$.
  }  \label{Fig_Global density}
\end{figure}

\begin{rem}[\textbf{Leading order asymptotic of the global densities}]
We stress that the leading order of the density at strong non-Hermiticity (i.e. the $1/(2(1+\tau))$ part in \eqref{rhoN asymp sH}) was obtained in \cite[Section 6.2]{FN08}.
Note that the uniform density is quite natural given the elliptic law \eqref{ellipse}.
We refer to \cite{Ta22} for the ``square root'' relation between the complex and real eigenvalue densities in a more general context. 

On the other hand, the leading order of the density at weak non-Hermiticity (i.e. the $\rho_{(0)}$ part of \eqref{rhoN asymp wH}) was shown in \cite[Theorem 2.8]{BKLL23}.
This formula was indeed predicted by Efetov \cite{Efe97} some time ago now, utilizing the supersymmetry method in the context of directed quantum chaos. 
A notable feature of this one-parameter family of densities is that it interpolates the uniform density (cf. \eqref{rho 0 1 sH} with $\tau=1$) and the Wigner's semi-circle law:
\begin{equation}
\rho_{(0)}^{ \rm w }(x) \sim 
\begin{cases}
\displaystyle \frac{1}{4} & \textup{ as } \alpha \to \infty; 
\smallskip 
\\
\displaystyle  \frac{\sqrt{4-x^2}}{2\pi}  & \textup{ as } \alpha \to 0.
\end{cases}
\end{equation}

Note that, by \eqref{EN RN int}, these leading order asymptotic behaviours yield the leading order asymptotic behaviours of the expected number of real eigenvalues given in \eqref{EN tau sH} and \eqref{EN tau wH}. 
For the weakly non-Hermitian regime, one also needs the integral representation
\begin{equation} \label{c(alpha) int rep}
c(\alpha)= \frac{2}{\alpha \sqrt{\pi}} \int_{0}^{1} \erf( \alpha \sqrt{1-s^2} ) \,ds
\end{equation}
of the function $c(\alpha)$ in \eqref{c(alpha)}, see e.g. \cite{BKLL23,BMS23}. 
\end{rem}

\begin{rem}[\textbf{Global density of the GinOE}] \label{Rem_Global GinOE}
For the GinOE case when $\tau=0$, it can be shown that the $1$-point function $\bfR_N$ satisfies the asymptotic behaviour
\begin{equation} \label{RN GinOE delta}
\bfR_N(x) \Big|_{\tau = 0} = \Big( \frac{N}{ 2\pi } \Big)^{ \frac12 }\, \mathbbm{1}_{(-1,1)}(x)+  \frac14\Big( \delta(x-1)+\delta(x+1) \Big) + O( N^{-\frac12} ),
\end{equation}
in the sense of distribution, see \cite[pp.\,33--34]{BF23}. 
In connection with the formula \eqref{RN asymp sH} for $\tau=0$, one can observe that there are no additional polynomial order contributions from the bulk of the spectrum.
The asymptotic formula \eqref{RN GinOE delta} follows from the perfect screening of the charge density in the Coulomb gas picture of the GinOE as well as the edge scaling limit \eqref{R0 edge s} below.
This in turn implies that the $O(1)$-term of the formula \eqref{RN GinOE delta} remains valid for general value of $\tau \in [0,1)$ at strong non-Hermiticity.
Note also that by \eqref{EN RN int}, the expansion \eqref{RN GinOE delta} indeed gives rise to the $1/2$ in the $O(1)$-term of the expansion \eqref{EN tau sH} for $\tau=0$.
This formula can also be used to derive the large-$N$ expansion of the moment generating function, which satisfies a linear differential equation recently found in \cite[Corollary 1.5]{By23b}. 
We also refer to \cite{GJ21,Ja23} for the rate of convergence of the empirical measure of the GinOE.
\end{rem}

\bigskip

Next, we investigate the edge scaling limits of the real eigenvalue densities. 
It is worth recalling here that, due to the elliptic law \eqref{ellipse}, the edge of real eigenvalues is located at $\pm (1+\tau).$ 
Furthermore, by the symmetry $x \mapsto -x$, it suffices to consider the right-most edge $1+\tau.$

To describe the edge scaling limits at weak non-Hermiticity, it is convenient to use the notation
\begin{equation}
    \Airy_\alpha(x) \coloneqq e^{\frac{\alpha^6}{12}+\frac{\alpha^2}{2}x} \Airy(x+\tfrac{\alpha^4}{4}),
\end{equation}
where 
\begin{equation}
\Airy(x):= \frac{1}{\pi} \int_0^\infty \cos\Big( \frac{t^3}{3}+xt \Big)\,dt
\end{equation}
is the Airy function, see e.g. \cite[Chapter 9]{NIST}. 
Furthermore, we define the rescaling densities.
\begin{itemize}
    \item \textbf{Strong non-Hermiticity}. For a fixed $\tau \in [0,1)$, define 
 \begin{equation} \label{rescaled RN sH}
    R_N^{ \rm s }(\xi):=    \sqrt{ \frac{1-\tau^2}{N} }  \bfR_N\Big(1+\tau+  \sqrt{ \frac{1-\tau^2}{N} }\, \xi \Big) .
    \end{equation}
    \item \textbf{Weak non-Hermiticity}.  For $\tau_N = 1 - \frac{\alpha^2}{N^{1/3}}$ with fixed $\alpha \in [0,\infty)$, define 
    \begin{equation} \label{rescaled RN wH}
    R_N^{ \rm w }(\xi):= \frac{1}{N^{2/3}} \bfR_N\Big(1+\tau + \frac{\xi}{N^{2/3}}\Big).
    \end{equation}
\end{itemize}
We mention that the rescaling orders are chosen according to the typical eigenvalue spacings. 
Then we obtain the following theorem. 

\begin{theorem}[\textbf{Finite size correction of the edge scaling densities}]  \label{Thm_Edge density asym}
Let $N$ be an even integer. Then we have the following.
\begin{itemize}
    \item[\textup{(i)}] \textup{\textbf{(Strong non-Hermiticity)}} Let $\tau \in [0,1)$ be fixed. Then as $N \to \infty$, we have 
    \begin{equation} \label{rescaled RN asymp sH}
    R_N^{ \rm s }(\xi) = R_{(0)}^{\rm s}(\xi)+ R_{ (1) }^{\rm s} (\xi) \, N^{-\frac12} +O( N^{-1} ),
    \end{equation}
   uniformly on compact subsets of $\R$, where
      \begin{align}
    R_{(0)}^{ \rm s}(\xi) &:= \frac{1}{2\sqrt{2\pi} } \Big( 1 - \erf(\sqrt{2}\xi)  + \frac{e^{-\xi^2}}{ \sqrt{2} } ( 1 + \erf(\xi) ) \Big), \label{R0 edge s}
    \\
    R_{(1)}^{ \rm s}(\xi) &:= \frac{ \sqrt{1-\tau^2} }{12\pi(1-\tau)^2} e^{-2\xi^2} \Big( (1+\tau)\xi^2 - 3 \Big) \Big( 1+ e^{\xi^2}\sqrt{\pi}\xi \big( 1+\erf(\xi) \big) \Big).  \label{R1 edge s}
    \end{align} 
    
     \item[\textup{(ii)}] \textup{\textbf{(Weak non-Hermiticity)}}
    Let $\tau = 1 - \alpha^2/N^{\frac{1}{3}}$ with fixed $\alpha \in [0,\infty)$. Then as $N \to \infty,$ we have
    \begin{align}
    \begin{split}
    R_N^{ \rm w }(\xi) =  R_{(0)}^{ \rm w }(\xi) +R_{(1)}^{ \rm w }(\xi) \,  N^{ -\frac{1}{3} }  + O(N^{ -\frac23 + \epsilon }),
    \end{split}
    \end{align}
    uniformly on compact subsets of $\R$ for any $\epsilon>0$, where
    \begin{align}
    R_{(0)}^{ \rm w } (\xi)&:= \int_\xi^\infty \Airy_\alpha(t)^2\,dt + \frac{1}{2} \Airy_\alpha(\xi) \bigg( 1 - \int_\xi^\infty \Airy_\alpha(t)\,dt \bigg), \label{R0 edge w}
    \\
    \begin{split}
    R_{(1)}^{ \rm w } (\xi) &:= -\frac{\alpha^6+2\alpha^2\xi+2}{8} \Airy_\alpha(\xi)^2 + \int_\xi^\infty \bigg(\frac{\alpha^4t+\alpha^2}{2} \Airy_\alpha(t)^2 - \frac{\alpha^4}{8}  \Airy_\alpha(\xi) \,t \Airy_\alpha(t) \bigg)\,dt
    \\
    & \quad + \bigg(\frac{\alpha^4\xi+2\alpha^2}{8} \Airy_\alpha(\xi) + \frac{2\alpha^2\xi+\alpha^6+2}{8}\Airy_\alpha'(\xi) \bigg) \bigg( 1 - \int_\xi^\infty \Airy_\alpha(t) \,dt \bigg).    \label{R1 edge w}
    \end{split}
    \end{align}
\end{itemize}
\end{theorem}

See Figure~\ref{Fig_RN edge} for the numerics on Theorem~\ref{Thm_Edge density asym}.

\begin{figure}[h!]
	\begin{subfigure}{0.45\textwidth}
		\begin{center}
			\includegraphics[width=\textwidth]{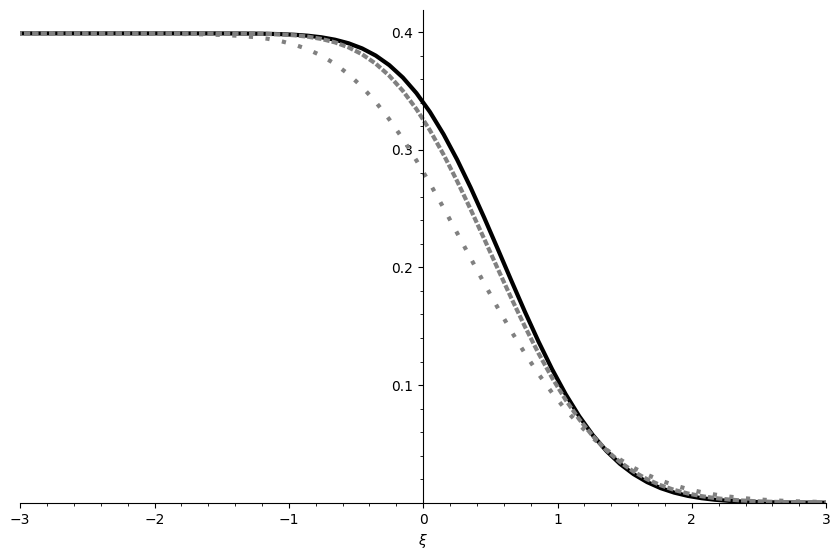}
		\end{center}
		\subcaption{ $R_N$ and $R_{(0)}^{ \rm s} $ }
	\end{subfigure}
   	\begin{subfigure}{0.45\textwidth}
		\begin{center}
			\includegraphics[width=\textwidth]{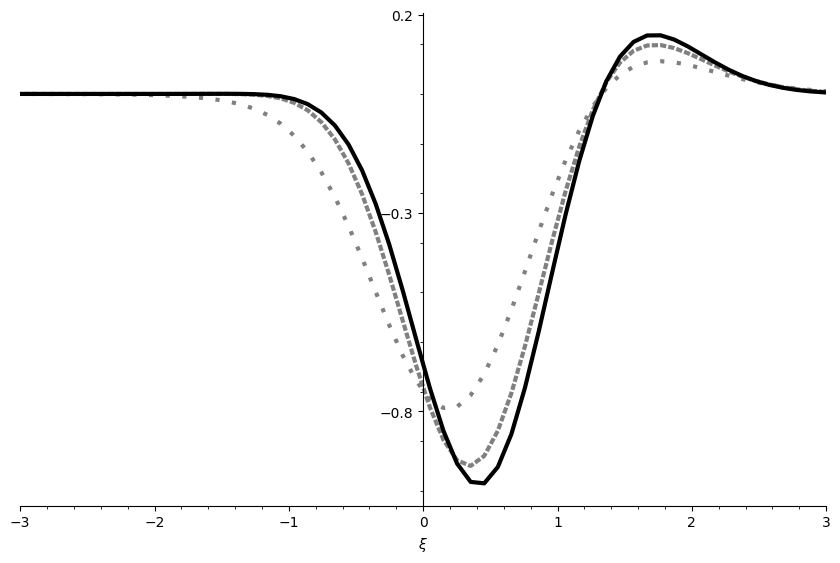}
		\end{center}
		\subcaption{  $N^{1/2}(R_N-R_{(0)}^{ \rm s}) $ and $R_{(1)}^{ \rm s}$  }
	\end{subfigure}

         \begin{subfigure}{0.45\textwidth}
		\begin{center}
			\includegraphics[width=\textwidth]{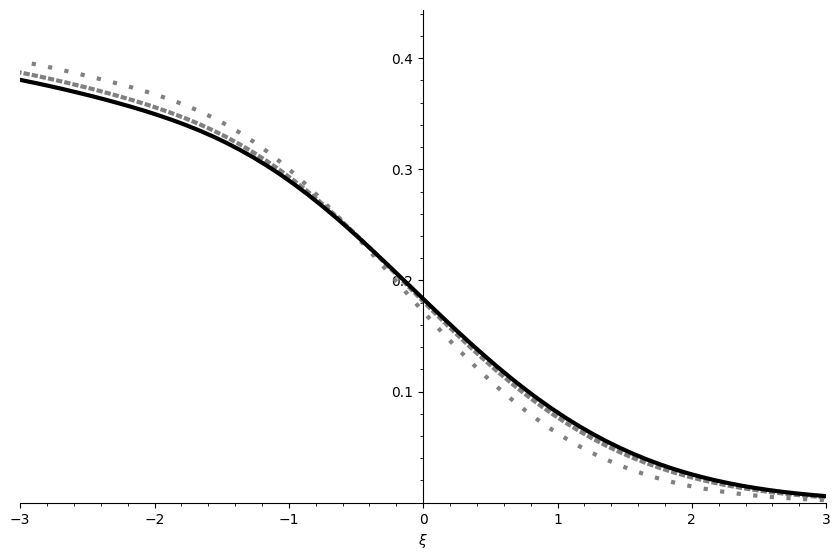}
		\end{center}
		\subcaption{ $R_N$ and $R_{(0)}^{ \rm w} $ }
	\end{subfigure}
   	\begin{subfigure}{0.45\textwidth}
		\begin{center}
			\includegraphics[width=\textwidth]{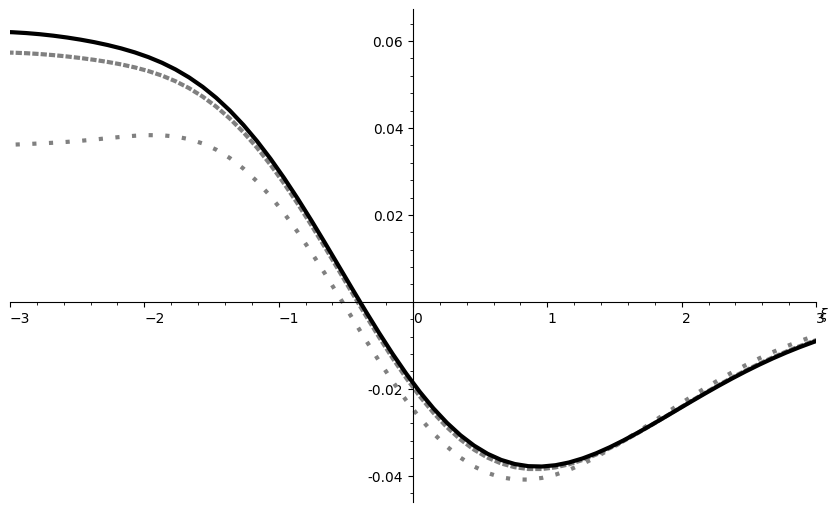}
		\end{center}
		\subcaption{  $N^{1/3}(R_N-R_{(0)}^{ \rm w}) $ and $R_{(1)}^{ \rm w}$  }
	\end{subfigure}
	\caption{ The plot (A) shows the function $R_N$ converging to the limiting density $R_{(0)}^{\rm s}$ in \eqref{rescaled RN sH} for $N = 160,\ 2560$ with $\tau = 5/7$. The plot (B) is the graph of $R_{(1)}^{\rm s}$ and its comparison to $N^{1/2}(R_N - R_{(0)}^{\rm s})$ with the same $N$'s and $\tau$. The plot (C) and (D) are analogous figures for the weak non-Hermiticity regime with $N = 10,\ 640$ and $\alpha = 2/3$. } \label{Fig_RN edge}  
\end{figure}

\begin{rem}[\textbf{Edge scaling limits and universality}] \label{Rem_edge universality}
The leading order edge scaling limit \eqref{R0 edge s} was initially derived for the GinOE case ($\tau=0$) in \cite{FN07,BS09} and subsequently extended to encompass general values of $\tau \in [0,1)$ in \cite{FN08}.
It is worth noting that the function $R_{(0)}^{\text{ s }}$ remains invariant regardless of the chosen $\tau$, illustrating the universality principle in random matrix theory. We refer the reader to \cite{Ku11} for a general review on the universality and also \cite{CES21,HW21} and references therein for universality of non-Hermitian random matrices. 
In contrast to the leading order, the sub-leading correction term $R_{(1)}^{\rm{ s }}$ does exhibit dependence on $\tau$, making it non-universal.
For the elliptic GinUE and GinSE models, the subleading corrections were obtained in \cite{LR16} and \cite{BE22} respectively. In those works, a similar dependence on $\tau$ is observed.

For the weakly non-Hermitian regime, the leading order edge scaling limit \eqref{R0 edge w} was obtained in \cite{AP14}, where the authors introduced the non-Hermitian extension of the classical Airy point process. 
In this case, the leading order already depends on $\tau$ (or $\alpha$). One may interpret that in this case, the choice of $\alpha$ determines the intrinsic geometric property—the curvature of the ellipse \eqref{ellipse} that reveals a critical transition—and thus determines the universality class.
\end{rem}

\begin{rem}[\textbf{Finite size corrections of the GOE}] \label{Rem_GOE density}
Let us discuss some immediate consequences of Theorems~\ref{Thm_Bulk density} and \ref{Thm_Edge density asym} for the GOE case.
For the case $\tau=1$ (i.e. $\alpha=0$), by Theorem~\ref{Thm_Bulk density} (ii) with $\alpha=0$, we have 
 \begin{equation} \label{GOE density correction}
        \rho_N(x) \Big|_{\alpha=0} = \frac{\sqrt{4-x^2}}{2\pi} -  \frac{1}{2\pi\sqrt{4-x^2}}\,N^{-1} + O(N^{-2}),
    \end{equation}
where we have used the asymptotic behaviour of the error function 
$$
\erf(x) \sim \frac{2}{\sqrt{\pi}} \, x, \qquad (x \to 0),
$$
see e.g. \cite[Eq.(7.6.1)]{NIST}.
This coincides with the finite size correction of the GOE density towards the Wigner's semi-circle law established in \cite[Proposition 5]{FFG06}.
We also refer to \cite{It97} for the resolvent approach to derive \eqref{GOE density correction}. 

For the edge scaling density, it follows from Theorem~\ref{Thm_Edge density asym} (ii) with $\alpha=0$ that 
\begin{align}
\begin{split}
\label{edge density GOE}
R_N(\xi) \Big|_{\alpha=0} &= R_{(0)}^{\rm w}(\xi) \Big |_{\alpha = 0} + \frac{1}{2} \Big ( \frac{d}{d \xi} R_{(0)}^{\rm w}(\xi) \Big |_{\alpha = 0} \Big ) N^{-1/3}+ O(N^{-\frac23+ \epsilon})
\\
& = \Airy'(\xi)^2 - \xi\, \Airy(\xi)^2 + \frac{1}{2} \Airy(\xi) \bigg( 1 - \int_\xi^\infty \Airy(t)\,dt \bigg)
\\
&\quad +  \frac{1}{4} \bigg[\Airy'(\xi) \bigg( 1 - \int_\xi^\infty \Airy(t) \,dt \bigg) -  \Airy(\xi)^2 \bigg] N^{-\frac{1}{3}}  + O( N^{ -\frac23 + \epsilon } ).
\end{split}
\end{align}
The fact that the correction in \eqref{edge density GOE} can be written as a derivative implies that there is a re-centering of $\xi$ on the left hand side, specifically $\xi \mapsto \xi - 1/(2N^{1/3})$, which
eliminates the leading correction on the right hand side. 
This is in keeping with the optimal rate of convergence being $O(N^{-2/3})$
for edge scaling of the GOE \cite{JM12}.  

We mention that in contrast to the above, the correction term in \eqref{edge density GOE} does not entirely coincide with \cite[Proposition 9]{FFG06}. 
Indeed, as confirmed by the authors, there is a minor typo in \cite[Proposition 9]{FFG06} that arises from a typo in the Plancherel-Rotach asymptotic formula in \cite[Eq.(3.14)]{FFG06}, where $(2N)^{1/2}$ should be corrected to $(2N+1)^{1/2}$. 
With this correction, the resulting formula coincides with our formula \eqref{edge density GOE}.
\end{rem}

\begin{rem}[\textbf{Further expansions}]
We note that the methods employed for our main results can also be extended to derive more precise expansions in Theorems~\ref{Thm_Bulk density} and \ref{Thm_Edge density asym}. 
To achieve this, more detailed asymptotic behaviours of the Hermite polynomials (Plancherel-Rotach formulas) are required. 
We mention that in some of our analysis, we have already crucially utilized a recent work \cite{SNWW23} on such behaviours to derive the first subleading corrections.
On one hand, in Theorem~\ref{Thm_Bulk density}, it is also possible to deduce the asymptotic behaviour of the $1$-point function $\bfR_N$ outside its support (i.e. $x \not \in (-1-\tau,1+\tau)$), unveiling the exponential decays.
\end{rem}

\subsection*{Organisation of the paper}
The rest of this paper is organised as follows. 
In Section~\ref{sec_Pre}, we provide the preliminaries necessary to prove our main results. Specifically, we revisit the skew-orthogonal polynomial representation of real eigenvalue densities and present the Plancherel-Rotach asymptotic formulas for the Hermite polynomials.
In Section~\ref{sec_global}, we study the real eigenvalue densities in the global regime and prove Theorem~\ref{Thm_Bulk density}.
In Section~\ref{sec_edge}, we investigate the edge scaling limits and establish Theorem~\ref{Thm_Edge density asym}.

\subsection*{Acknowledgements} The authors gratefully acknowledge Peter J. Forrester for careful reading of the previous version of the paper, as well as for valuable suggestions. 
In particular, we thank him for providing us with the derivative expression of the correction term in \eqref{edge density GOE} and for confirming a typo in \cite{FFG06}.
We also express our gratitude to Gernot Akemann for his interest and helpful discussions.

\section{Preliminaries} \label{sec_Pre}
In this section, we review some standard facts on the real eigenvalue density of the elliptic GinOE. 
We then recall the Plancherel-Rotach asymptotic formulas of the Hermite polynomials that will be crucially used in the later analysis.
Unless otherwise stated, we assume that $N$ is an even integer and $\tau \in (0, 1]$. 

It is well known that the elliptic GinOE eigenvalues form a Pfaffian point process, and its $2 \times 2$ matrix-valued kernel can be expressed in terms of the associated skew-orthogonal polynomials, see e.g. \cite[Section 2.8]{BF23}.
Furthermore, in \cite{FN08}, the associated skew-orthogonal polynomial is constructed in terms of the classical Hermite polynomial 
$$
H_k(x) := (-1)^k e^{x^2} \frac{d}{dx} e^{-x^2}.
$$
As a result, it was shown in \cite[Section 6.2]{FN08} that after some manipulations, the $1$-point function $\bfR_N$ can be written as
\begin{equation} \label{RN = RN1 + RN2}
    \bfR_N(x) = \bfR_N^1(x) + \bfR_N^2(x),
\end{equation}
where
\begin{align} \label{eq. RN1 discrete summation}
    \bfR_N^1(x) &= \sqrt{\frac{N}{2\pi}} e^{-\frac{N}{1+\tau}x^2} \sum_{k=0}^{N-2} \frac{(\tau/2)^k}{k!} H_k(\sqrt{\tfrac{N}{2\tau}}x)^2, 
    \\
    \begin{split} \label{eq. RN2 integral representation}
    \bfR_N^2(x) & = \frac{1}{\sqrt{2\pi}} \frac{(\tau/2)^{N-\frac{3}{2}}}{1+\tau} \frac{N}{(N-2)!} e^{-\frac{N}{2(1+\tau)}x^2}H_{N-1}(\sqrt{\tfrac{N}{2\tau}}x) \int_0^x e^{-\frac{N}{2(1+\tau)}u^2}H_{N-2}(\sqrt{\tfrac{N}{2\tau}}u) \,du.
   \end{split}
\end{align} 
See also \cite[Appendix A]{By23b}.
We first discuss the $\tau=0$ case.

\begin{rem} \label{Rem_tau=0}
    For the GinOE case $\tau=0$, the 1-point function $\bfR_N$ was obtained in \cite[Corollary 4.3]{EKS94} without the use of the skew-orthogonal polynomial formalism.
    In this case, the Hermite polynomials in \eqref{eq. RN1 discrete summation} and \eqref{eq. RN2 integral representation} are replaced by the monomials, which yields the expression 
    \begin{equation} \label{RN for tau=0}
        \bfR_N(x) \Big|_{\tau = 0} = \sqrt{\frac{N}{2\pi}} \Big( 1 - \frac{\gamma(N-1, N x^2)}{\Gamma(N-1)} + \frac{(2N)^{\frac{N-1}{2}}}{2\, \Gamma(N-1)} |x|^{n-1}e^{-\frac{N}{2}x^2} \gamma(\tfrac{N-1}{2}, \tfrac{N}{2}x^2) \Big),
    \end{equation}
    where $\gamma(s, x) := \int_0^x t^{s-1} e^{-t} \, dt$ is the lower incomplete gamma function.
    Let us mention that the expression \eqref{RN for tau=0} holds not only for the even integer $N$ but also for the odd integer $N$. 
   Due to the expression \eqref{RN for tau=0}, the analysis of $\bfR_N$ is straightforward and follows from the uniform asymptotic behaviours of the incomplete gamma function \cite[Section 11.2.4]{Te96}.
\end{rem}

From the above discussion, one can realize that the analysis of $\bfR_N$ highly relies on whether one can find an effective way to analyse the summation in \eqref{eq. RN1 discrete summation}. 
For the GOE case when $\tau=1$, this can be achieved using the classical Christoffel-Darboux formula 
$$
\sum_{k=0}^{N-2} \frac{1}{2^k\,k!} H_k(x)^2=  \frac{1}{2} \Big( H_{N-2}(x) H_{N-1}'(x) - H_{N-2}'(x) H_{N-1}(x) \Big).
$$
This allows to derive the large-$N$ behaviour of the $1$-point function for the GOE, see \cite{FFG06}.
In contrast, for general $\tau \in (0,1)$, the Christoffel-Darboux formula cannot be applied. 
Instead, an effective way to analyse $\bfR_N^1$ was found in \cite[Section 4]{BKLL23} through the observation
\begin{equation}
\Big( \bfR_N^1(x) \Big)' = - \sqrt{\frac{2}{\pi}} \frac{(\tau/2)^{N-\frac{3}{2}}}{1+\tau} \frac{N}{(N-2)!}    e^{-\frac{N}{1+\tau}u^2} H_{N-2}(\sqrt{\tfrac{N}{2\tau}}x) H_{N-1}(\sqrt{\tfrac{N}{2\tau}}x). 
\end{equation}
This follows from the generalised Christoffel-Darboux formula introduced by Lee and Riser in \cite{LR16}. 
Then we have the following lemma.

\begin{lemma}[\textbf{Integral representation of $\bfR_N^1$}]
The function $\bfR_N^1$ in \eqref{eq. RN1 discrete summation} has an integral representation 
\begin{equation} \label{eq. RN1 integral representation}
    \bfR_N^1(x) = \bfR_N^1(x_0) - \sqrt{\frac{2}{\pi}} \frac{(\tau/2)^{N-\frac{3}{2}}}{1+\tau} \frac{N}{(N-2)!}  \int_{x_0}^x e^{-\frac{N}{1+\tau}u^2} H_{N-2}(\sqrt{\tfrac{N}{2\tau}}u) H_{N-1}(\sqrt{\tfrac{N}{2\tau}}u) \,du,
\end{equation}
for any $x_0\in\R \cup \{\infty \}$.
\end{lemma}

One can make use of the expression \eqref{eq. RN1 integral representation} to avoid dealing with the summation in \eqref{eq. RN1 discrete summation}. Indeed, analysing such a summation becomes particularly involved when considering the precise asymptotic expansion.
On the other hand, additional difficulties arise from the need to analyse the extra integral in \eqref{eq. RN1 integral representation}. 
As discussed later, this in turn requires analysing certain oscillatory integrals. 
Furthermore, one needs to derive the asymptotic value of the initial term $\bfR_N^1(x_0)$ by examining the summation \eqref{eq. RN1 discrete summation}.
To address this, we will choose a suitable value for $x_0$ according to the situation under consideration.

\medskip 

After these preparations, one can expect that the asymptotic behaviours of the Hermite polynomials will play a key role in the later analysis. 
These are known as the Plancherel-Rotach formulas, see e.g. \cite{SNWW23,Sz75} and references therein.
To describe the Plancherel-Rotach formulas in various regimes, it is convenient to use the following notations:
\begin{align}
\phi(x) &:= x\sqrt{1-x^2} - \arccos(x); \label{phi}
\\
\theta_m(x) &:= (m+\frac{1}{2}) \arccos(x) - \frac{\pi}{4}; \label{theta_m}
\\
\sigma(x) &:= x+\sqrt{x^2-1}. \label{sigma}
\end{align}

\begin{lemma}[\textbf{The Plancherel-Rotach formula for the oscillatory regime}]
\label{Lem_Plancherel-Rotach osc}
    Let $N$ and $m$ be integers, and $x \in (-1,1)$.
    Then as $N \to \infty$, we have
    \begin{align} \label{eq. Plancherel oscillatory regime}
    \begin{split}
        H_{N+m}(\sqrt{2N}x) &= \frac{1}{\sqrt{\pi}(1-x^2)^{\frac{1}{4}}} (N+m)! e^{\frac{N}{2}} 2^{\frac{N+m}{2}} N^{-\frac{N+m+1}{2}} e^{Nx^2}
        \\
        & \quad \times \Big( h_{(0)}^{ \rm osc }(x) + h_{(1)}^{ \rm osc }(x) N^{-1} + O(N^{-2}) \Big),
    \end{split}
    \end{align}
    where
    \begin{align}
    \begin{split}
        h_{(0)}^{ \rm osc }(x) &:= \cos( N\phi(x) - \theta_m(x) ),
        \\
        h_{(1)}^{ \rm osc }(x) &:= A_m(x) \cos(N\phi(x)-\theta_m(x)) + B_m(x) \sin(N\phi(x)-\theta_m(x)).
    \end{split}
    \end{align}
    Here, $\phi$ and $\theta_m$ are given in \eqref{phi} and \eqref{theta_m}, and
    \begin{align} \label{coeff Am Bm}
    \begin{split}
        A_m(x) &:= \frac{(6m^2+6m+1)x^2-6m^2-12m-4}{24(1-x^2)},
        \\
        B_m(x) &:= -\frac{(12m^2+12m+2)x^3-(12m^2+12m-3)x}{48(1-x^2)^{\frac{3}{2}}}.
    \end{split}
    \end{align}
\end{lemma}

\begin{lemma}[\textbf{The Plancherel-Rotach formula for the critical regime}]
\label{Lem_Plancherel-Rotach crit}
    Let $N$ and $m$ be integers, and $\xi \in \R$.
    Then as $N \to \infty$, we have
    \begin{equation}
    e^{-N x^2} H_{N+m}(\sqrt{2N} x) = \pi^{\frac{1}{4}}2^{\frac{N}{2}+\frac{1}{4}}(N!)^{\frac{1}{2}} N^{-\frac{1}{12}} \Big( \Airy(\xi) - (m+\tfrac{1}{2})\Airy'(\xi) N^{-\frac{1}{3}} + O(N^{-\frac{2}{3}}) \Big),
    \end{equation}
    for $x = 1 + 2^{-1} N^{-\frac{2}{3}} \xi $, where $\Airy$ is the Airy function.
\end{lemma}

\begin{lemma}[\textbf{The Plancherel-Rotach formula for the exponential regime}]
\label{Lem_Plancherel-Rotach exp}
    Let $N$ and $m$ be integers, and $x\in(1,\infty)$.
    Then as $N\to\infty$, we have
    \begin{equation} \label{eq. Plancherel exponential regime}
        H_{N+m}(\sqrt{2N}x) = \frac{1}{2\sqrt{\pi} (x^2-1)^\frac{1}{4}} e^{-\frac{N}{2}}  2^\frac{N+m}{2} N^{\frac{N+m-1}{2}} \sigma(x)^{N+m+\frac{1}{2}} e^{Nx/\sigma(x)}  \Big( 1 + O(N^{-1})\Big),
    \end{equation}
    where $\sigma$ is given in \eqref{sigma}.
\end{lemma}

Note that in the analysis of \eqref{RN = RN1 + RN2}, we shall use the asymptotic behaviours of the Hermite polynomial of the form $H_{N+m}(\sqrt{2N} \frac{x}{ 2\sqrt{\tau} }). $
On the one hand, notice here that the foci of the ellipse \eqref{ellipse} are located at $\pm 2\sqrt{\tau}$.  
Therefore, depending on the position of $x$ with respect to the foci, we need to apply the Plancherel-Rotach formula in different regimes:
\begin{itemize}
    \item if $|x|<2\sqrt{\tau}$, it corresponds to the oscillatory regime, Lemma~\ref{Lem_Plancherel-Rotach osc}, cf. Region I in Figure~\ref{Fig_PR regimes}; 
    \smallskip 
    \item if $|x \pm 2\sqrt{\tau}|=O(N^{-2/3})$, it corresponds to the critical regime, Lemma~\ref{Lem_Plancherel-Rotach crit}, cf. Region II in Figure~\ref{Fig_PR regimes};
    \smallskip 
    \item if $|x|>2\sqrt{\tau}$, it corresponds to the exponential regime, Lemma~\ref{Lem_Plancherel-Rotach exp}, cf. Region III in Figure~\ref{Fig_PR regimes}. 
\end{itemize}
(Let us mention that the line segment connecting two foci coincides with the limiting zero set of the planar Hermite polynomial known as the limiting skeleton.)
Note also that in the weakly non-Hermitian regime where $\tau \uparrow 1$, the focus $2\sqrt{\tau}$ becomes close to the right-most point $1+\tau$ of the ellipse. 
In summary, depending on the situations, we will use the following strategy.

\begin{itemize}
\item \textbf{Global and strongly non-Hermitian regime, Theorem~\ref{Thm_Bulk density} (i).} In this case, we put $x_0=0$. Then it is required to apply the asymptotic formulas for all Regions I, II, and III. 
Indeed, we shall use an estimate \eqref{eq. rough bound of Hermite polynomial} that allows us avoid applying the formulas in Regions I and II. 
\smallskip
\item \textbf{Global and weakly non-Hermitian regime, Theorem~\ref{Thm_Bulk density} (ii).} Again, we choose $x_0=0$ and apply the asymptotic formulas in Region I.
\smallskip
\item \textbf{Edge scaling and strongly non-Hermitian regime, Theorem~\ref{Thm_Edge density asym} (i).} In this case, we put $x_0=+\infty$. Then it suffices to apply the asymptotic formula in Region III.
\smallskip
\item \textbf{Edge scaling and weakly non-Hermitian regime, Theorem~\ref{Thm_Edge density asym} (ii).} Again, we choose $x_0=+\infty$ and apply the asymptotic formulas in Regions II and III.
\end{itemize}

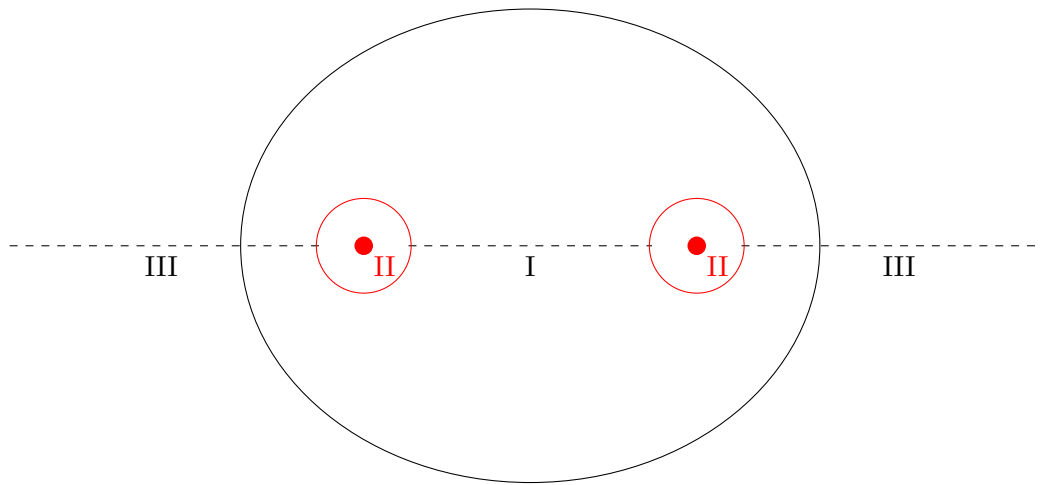
\begin{figure}[h!]
    \centering
\begin{tikzpicture}
    \def\t{0.1} 
    \def\scalefactor{3.5} 
    \begin{scope}[scale=\scalefactor]
        \draw[rotate=0] (0,0) ellipse ({1+\t} and {1-\t});
        
        \draw[dashed] ({2*sqrt(\t)+0.17},0) -- (2,0) node[midway, below] {III};

         \draw[dashed] ({-2*sqrt(\t)-0.17},0) -- (-2,0) node[midway, below] {III};
        
         \draw[dashed] ({-2*sqrt(\t)+0.17},0) -- ({2*sqrt(\t)-0.17},0) node[midway, below] {I};
         
         \draw[red] ({2*sqrt(\t)},0) circle (0.18) node[below right] {II};
        \draw[red] ({-2*sqrt(\t)},0) circle (0.18) node[below right] {II};
        \fill[red] ({2*sqrt(\t)}, 0) circle (1pt);
         \fill[red] ({-2*sqrt(\t)}, 0) circle (1pt);
    \end{scope}
\end{tikzpicture}
    \caption{Illustration of the regions to which different forms of the Plancherel-Rotach formula apply. Here, the dots indicate two foci of the ellipse \eqref{ellipse}.}
    \label{Fig_PR regimes}
\end{figure}

\section{Global real eigenvalue densities} \label{sec_global}
In this section, we prove Theorem~\ref{Thm_Bulk density}.
We shall divide the proof into two lemmas which derive the asymptotic behaviours of $\bfR_N^1$ and $\bfR_N^2$ respectively.
Then, by \eqref{RN = RN1 + RN2}, the main assertions follow from the lemmas.
While the finite size corrections in \eqref{RN asymp sH} and \eqref{RN asymp wH} differ qualitatively, the overall strategies for deriving these corrections are similar. The outline of the proof can be summarised as follows: 
\begin{itemize}
    \item we first analyse the asymptotic behaviour of $\bfR_N^1(0)$ in Lemma~\ref{Lem_RN1(0) sH} (\textit{resp.}, Lemma~\ref{Lem_RN1(0) wH}) using the properties of the hypergeometric functions;
\smallskip 
    \item using the asymptotic behaviours of the Hermite polynomials \eqref{eq. rough bound of Hermite polynomial} and \eqref{eq. Plancherel oscillatory regime (tau fixed)}, we then derive the asymptotic behaviour of $\bfR_N^1(x)$ in Lemma~\ref{Lem_RN1(x) sH} (\textit{resp.}, Lemma~\ref{Lem_RN1(x) wH});
\smallskip 
    \item similarly, we derive the asymptotic behaviour of $\bfR_N^2(x)$ in Lemma~\ref{Lem_RN2(x) sH} (\textit{resp.}, Lemma~\ref{Lem_RN2(x) wH}).
\end{itemize}

\subsection{Strong non-Hermiticity}
In this subsection, we prove Theorem~\ref{Thm_Bulk density} (i).
We assume throughout this subsection that $\tau\in(0,1)$ is fixed.
For $\tau=0$, see Remark~\ref{Rem_tau=0}.

We begin with the analysis of the large-$N$ behaviour of $\bfR_N^1(0)$.
For this purpose, recall that the regularized hypergeometric function $\gaussF(a,b;c;z)$ is defined by
\begin{equation} \label{regularized 2F1}
    \gaussF(a,b;c;z) := \frac{1}{\Gamma(a)\Gamma(b)} \sum_{s=0}^{\infty} \frac{\Gamma(a+s)\Gamma(b+s)}{\Gamma(c+s) \ s!}z^s, \qquad (|z|<1)
\end{equation}
and by the analytic continuation for $|z| \ge 1$, see e.g. \cite[Chapter 15]{NIST}.

\begin{lemma} \label{Lem_RN1(0) sH}
    Let $\tau\in(0,1)$ be fixed.
    Then as $N\to\infty$, we have
    \begin{equation}
    \bfR_N^1(0) = \Big(\frac{N}{2 \pi (1-\tau^2)}\Big)^\frac{1}{2} + O(e^{-\epsilon N}),
    \end{equation}
    for some $\epsilon>0$.
\end{lemma}
\begin{proof}
    By using \eqref{eq. RN1 discrete summation}, \eqref{regularized 2F1} and the  Hermite number
    $$
    H_k(0) = \frac{\sqrt{\pi}}{\Gamma(\frac{1-k}{2})} 2^k,
    $$ 
    we obtain
    \begin{align}
    \begin{split} \label{computation RN1(0)}
        \bfR_N^1(0) &= \Big( \frac{N}{2\pi} \Big)^{\frac12} \sum_{k=0}^{N-2} \frac{(\tau/2)^k}{k!} H_k(0)^2 = \Big(  \frac{N}{2} \pi  \Big)^{\frac12} \sum_{k=0}^{N-2}  \frac{(2\tau)^k}{k!} \frac{1}{ \Gamma(\frac{1-k}{2})^2 } 
        \\
        &= \Big(\frac{N}{2 \pi (1-\tau^2)}\Big)^\frac{1}{2} - \Big( \frac{N}{2\pi} \Big)^{\frac{1}{2}} \tau^N \frac{ \Gamma(\frac{N+1}{2}) }{ \sqrt{\pi} } \gaussF(1,\tfrac{N+1}{2};\tfrac{N+2}{2};\tau^2).
    \end{split}
    \end{align}
    We now  apply the Euler integral formula \cite[Eq.(15.6.1)]{NIST}: for $c>b>0$,
    \begin{equation} \label{NIST (15.6.1)}
        \gaussF(a,b;c;z)=\frac{1}{\Gamma(b)\Gamma(c-b)} \int_{0}^{1}\frac{t^{b-1}(1-t)^{c-b-1}}{(1-zt)^{a}} \,dt.
    \end{equation}
    This gives rise to 
    \begin{equation}
        \frac{ \Gamma(\frac{N+1}{2}) }{ \sqrt{\pi} } \gaussF(1,\tfrac{N+1}{2};\tfrac{N+2}{2};\tau^2)
        = \frac{1}{\pi} \int_0^1 \frac{ t^\frac{N-1}{2} (1-t)^{-\frac{1}{2}} }{ 1-\tau^2t } \,dt.
    \end{equation}
    Then it follows from 
    \begin{equation*}
        0 \leq \frac{t^{\frac{N-1}{2}}}{1-\tau^2 t} \leq \frac{1}{1-\tau^2}, \qquad t\in[0,1],
    \end{equation*}
    that
    \begin{equation*}
        0 \leq \frac{ \Gamma(\frac{N+1}{2}) }{ \sqrt{\pi} } \gaussF(1,\tfrac{N+1}{2};\tfrac{N+2}{2};\tau^2)        \leq  \frac{1}{\pi}\frac{2}{1-\tau^2}.
    \end{equation*}
    Since $\tau \in (0,1)$, one can observe that the second term in the last expression of \eqref{computation RN1(0)} decays exponentially as $N \to \infty$, and the proof is complete.
\end{proof}

It remains to show the exponential decay of the integrals in \eqref{eq. RN2 integral representation} and \eqref{eq. RN1 integral representation} as $N \to \infty$.
We first prove the exponential convergence of $\bfR_N^1$.
\begin{lemma}  \label{Lem_RN1(x) sH}
    Let $x \in [0, 1+\tau)$. Then as $N \to \infty$, we have
    \begin{equation}
        \mathbf{R}_N^1(x) = \mathbf{R}_N^1(0) + O(e^{-\epsilon N}),
    \end{equation}
    for some constant $\epsilon>0$.
\end{lemma}
\begin{proof}
    We first prove the lemma for $x \in [0, r_\tau]$, where
    \begin{equation} \label{r_tau & gamma_tau}
        r_\tau := \frac{2\sqrt{\tau} + \lambda_\tau}{2},
        \qquad
        \lambda_\tau := \Big(-\frac{ 2 \tau (1+\tau) \log\tau }{ 1-\tau } \Big)^{\frac{1}{2}}.
    \end{equation}
    We note here that $2\sqrt{\tau} < \lambda_\tau < 1 + \tau$, cf. Figure~\ref{Fig_PR regimes}. 
    The main ingredient of the proof is the inequality for the Hermite polynomials given in \cite[Eq.(8.954)]{GR14}.
    It says that there is a constant $C>0$ satisfying
    \begin{equation} \label{eq. rough bound of Hermite polynomial}
        \Big|H_{N+m}(\sqrt{\tfrac{N}{2\tau}}x)\Big| \leq C \Big((N+m)!\Big)^\frac{1}{2} e^\frac{Nx^2}{4\tau}, \qquad (N \to \infty),
    \end{equation}
    for any $x \in \R$.

    It follows from \eqref{eq. RN1 integral representation} and \eqref{eq. rough bound of Hermite polynomial} that
    \begin{align*}
        \big| \mathbf{R}_N^1(x) - \mathbf{R}_N^1(0) \big|
        \leq
        \widetilde{C} N^{\frac{3}{2}} \int_0^x e^{ N \omega_0(u) } \,du,
    \end{align*}
    for some constant $\widetilde{C}>0$, where
    \begin{equation} \label{omega_0}
        \omega_0(u) := \log\tau + \frac{1-\tau}{2\tau(1+\tau)} u^2.
    \end{equation}
    Note that $\omega_0(u)$ has two zeros $\pm\lambda_\tau$.
    Hence, it is negative and strictly increasing in $[0, r_\tau]$.
    Therefore, we have
    \begin{equation*}
        \int_0^x e^{ N \omega_0(u) } \,du
        \leq
         x e^{N \omega_0(x)} = O(e^{-\epsilon N}),
    \end{equation*}
    for some constant $\epsilon>0$.
    Therefore, the lemma is proved for $x \in [0, r_\tau]$.
    
    Next, we prove the lemma for $x \in [r_\tau,1+\tau)$.
    By applying \eqref{eq. Plancherel exponential regime} to \eqref{eq. RN1 integral representation}, we have
    \begin{equation*}
        \bfR_N^1(x) = \bfR_N^1(r_\tau) + \widehat{C} \int_{r_\tau}^x f(u) e^{N \omega_1(u)} \,du \ \Big( 1+O(N^{-1}) \Big),
    \end{equation*}
    for some constant $\widehat{C}>0$, where
    \begin{equation*}
        f(u) := (u^2-4\tau)^{-\frac{1}{2}}(u+\sqrt{u^2-4\tau})^{-2},
    \end{equation*}
    and
    \begin{equation}\label{omega_1}
        \omega_1(u) := \frac{ (1-\tau)u^2 - (1+\tau)u\sqrt{u^2-4\tau} }{ 2\tau(1+\tau) } + 2\log\Big( \frac{u+\sqrt{u^2-4\tau}}{2} \Big). 
    \end{equation}
    Notice that
    \begin{equation*}
        \omega_1'(u) = \frac{(1-\tau)u - (1+\tau)\sqrt{u^2-4\tau}}{\tau(1+\tau)}.
    \end{equation*}
    Therefore, we have $\omega_1'(1+\tau)=0$.
    Moreover, since 
    \begin{equation*}
        \omega_1''(u) = \frac{(1-\tau)\sqrt{u^2-4\tau} - (1+\tau){u}}{\tau(1+\tau)\sqrt{u^2-4\tau}}<0, \qquad u \in (2\sqrt{\tau}, 1+\tau],
    \end{equation*}
    the function $\omega_1(u)$ is increasing on $(2\sqrt{\tau},1+\tau]$.
    Furthermore, the function $f(u)$ is decreasing on $u \in (2\sqrt{\tau}, 1+\tau)$, which implies that $f(u)$ is bounded by $f(r_\tau)$ for any $u \in [r_\tau, 1+\tau)$.
    Thus, we have
    \begin{equation*}
        \int_{r_\tau}^x f(u) e^{N \omega_1(u)} \,du
        \leq
        (x - r_\tau) \, f(r_\tau) e^{N \omega_1(x)} \,du = O(e^{-\epsilon N}),
    \end{equation*}
    for some constant $\epsilon>0$. This completes the proof.
\end{proof}

It remains to prove the exponential decay of $\bfR_N^2$.
\begin{lemma} \label{Lem_RN2(x) sH}
    Let $x \in [0,1+\tau)$. Then as $N \to \infty$, we have
    \begin{equation}
        \mathbf{R}_N^2(x) = O(e^{-\epsilon N})
    \end{equation}
    for some constant $\epsilon>0$.
\end{lemma}
\begin{proof}
    Recall that $r_\tau$ and $\lambda_\tau$ are defined in \eqref{r_tau & gamma_tau}.
    We first prove the lemma for $x \in [0,r_\tau]$.
    It follows from \eqref{eq. RN2 integral representation} and \eqref{eq. rough bound of Hermite polynomial} that
    \begin{equation*}
        |\bfR_N^2(x)| \leq C N^{\frac{3}{2}} e^{ \frac{N}{2} \omega_0(x) }
        \int_0^x e^{ \frac{N}{2} \omega_0(u) } \,du
    \end{equation*}
    for some constant $C>0$, where $\omega_0(u)$ is given in \eqref{omega_0}.
    Similar to the above, one can observe that there exists $\epsilon>0$ such that
    \begin{equation*}
        e^{ \frac{N}{2} \omega_0(x) } \int_0^x e^{ \frac{N}{2} \omega_0(u) } \,du
        \leq
         x e^{N \omega_0(x)} = O(e^{-\epsilon N}),
    \end{equation*}
    for any $x \in [0, r_\tau]$.
    Thus, the lemma is proved for $x \in [0, r_\tau]$.
    
    Next, suppose $x \in [r_\tau, 1+\tau)$.
    By applying \eqref{eq. Plancherel exponential regime} to \eqref{eq. RN2 integral representation}, we have
    \begin{equation*}
        \bfR_N^2(x) = \bfR_N^2(r_\tau) + \widetilde{C} \, (x+\sqrt{x^2-4\tau}) g(x) e^{\frac{N}{2} \omega_1(x)}
        \int_{r_\tau}^x g(u) e^{\frac{N}{2} \omega_1(u)}  \,du \ \Big(1+O(N^{-1})\Big),
    \end{equation*}
    for some constant $\widetilde{C}>0$, where $\omega_1(u)$ is given in \eqref{omega_1} and
    \begin{equation*}
        g(u) := (u^2-4\tau)^{-\frac{1}{4}}(u+\sqrt{u^2-4\tau})^{-\frac{3}{2}}.
    \end{equation*}
    Then the rest of the proof proceeds in the same manner as the proof of Lemma~\ref{Lem_RN1(x) sH}.
\end{proof}

\subsection{Weak non-Hermiticity}

In this subsection, we prove Theorem~\ref{Thm_Bulk density} (ii).
As before, we first compute the asymptotic behaviour of $\bfR_N^1(0)$.
\begin{lemma} \label{Lem_RN1(0) wH}
    Let $\tau = 1 - \alpha^2/N$ with fixed $\alpha \in [0, \infty)$.
    Then as $N \to \infty$, we have
    \begin{equation}
        \bfR_N^1(0) = \frac{1}{2\alpha\sqrt{\pi}} \erf(\alpha) N + \frac{\alpha^2-1}{4 \pi} e^{-\alpha^2} + \frac{\alpha}{8 \sqrt{\pi}} \erf(\alpha) + O(N^{-1}).
    \end{equation}
\end{lemma}
\begin{proof}
    By \cite[Lemma 4.1]{BKLL23}, we have $\bfR_N^1(0) = a_N - b_N$, where
    \begin{equation*}
     a_N := \Big( \frac{N}{2\pi(1-\tau^2)} \Big)^{\frac{1}{2}},
        \qquad 
        b_N := \Big( \frac{N}{2\pi} \Big)^{\frac{1}{2}} \frac{\tau^N}{1-\tau^2} \frac{\Gamma(\tfrac{N+1}{2})}{\pi} \gaussF(1,\tfrac{1}{2}; \tfrac{N}{2}+1;\tfrac{\tau^2}{\tau^2-1}).
    \end{equation*}
   Notice here that 
    \begin{equation} \label{aN asymp}
        a_N = \frac{1}{2\alpha\sqrt{\pi}} N  + \frac{\alpha}{8\sqrt{\pi}} + O(N^{-1}).
    \end{equation}
    To analyse $b_N$, we make use of \cite[Eq.(15.8.3)]{NIST} and obtain 
    \begin{align*}
        \begin{split}
        \frac{\Gamma(\tfrac{N+1}{2})}{\pi} \gaussF(1,\tfrac{1}{2}; \tfrac{N}{2}+1;\tfrac{\tau^2}{\tau^2-1})
        &= \Big(\frac{\alpha^2(2N-\alpha^2)}{N^2}\Big)^{\frac{1}{2}} \gaussF(\tfrac{1}{2}, \tfrac{N}{2}; \tfrac{1}{2};\tfrac{\alpha^2(2N-\alpha^2)}{N^2})
        \\
        & \quad - \frac{ \Gamma(\tfrac{N+1}{2}) }{\sqrt{\pi}\Gamma(\frac{N}{2})} \frac{\alpha^2(2N-\alpha^2)}{N^2} \gaussF( 1, \tfrac{N+1}{2}; \tfrac{3}{2};\tfrac{\alpha^2(2N-\alpha^2)}{N^2}).
        \end{split}
    \end{align*}
    By using \eqref{NIST (15.6.1)}, we have
    \begin{align*}
        \gaussF( 1, \tfrac{N+1}{2}; \tfrac{3}{2};\tfrac{\alpha^2(2N-\alpha^2)}{N^2}) &= \frac{1}{\sqrt{\pi}} \int_0^1 \frac{1}{\sqrt{1-t}} \Big(1 - \frac{\alpha^2(2N-\alpha^2)}{N^2}t \Big)^{-\frac{N+1}{2}} \,dt.
    \end{align*}
    Note here that 
    \begin{equation*}
        \Big( 1 - \frac{\alpha^2 (2N-\alpha^2)}{N^2}t \Big)^{-\frac{N+1}{2}}
        = e^{\alpha^2 t} \Big(1 +  \big(\alpha^4t^2+\tfrac{(2-\alpha^2)\alpha^2}{2}t \big) N^{-1} + O(N^{-2}) \Big).
    \end{equation*}
    Thus, we obtain
    \begin{align*}
        \gaussF( 1, \tfrac{N+1}{2}; \tfrac{3}{2};\tfrac{\alpha^2(2N-\alpha^2)}{N^2}) &= \frac{1}{\sqrt{\pi}} \int_0^1 \frac{e^{\alpha^2 t}}{\sqrt{1-t}} \Big(1 +  \big( \alpha^4t^2+\tfrac{(2-\alpha^2)\alpha^2}{2}t \big) N^{-1} + O(N^{-2}) \Big) \, dt
        \\
        &= \frac{e^{\alpha^2}}{\alpha} \erf(\alpha) + \Big( \frac{\alpha^2 - 1}{2\sqrt{\pi}} + \frac{2\alpha^4 + \alpha^2 + 1}{4\alpha} e^{\alpha^2} \erf(\alpha) \Big) N^{-1} + O(N^{-2}).
    \end{align*}
    On the other hand, by \cite[Eq.(15.4.6)]{NIST} and the Taylor expansion, we have
    \begin{align*}
        \gaussF(\tfrac{1}{2}, \tfrac{N}{2}; \tfrac{1}{2};\tfrac{\alpha^2(2N-\alpha^2)}{N^2}) &= \frac{1}{\sqrt{\pi}} \Big( 1 - \frac{\alpha^2(2N-\alpha^2)}{N^2} \Big)^{-\frac{N}{2}} 
        = \frac{1}{\sqrt{\pi}} e^{\alpha^2} \Big( 1 + \frac{\alpha^4}{2} N^{-1} + O(N^{-2}) \Big).
    \end{align*}
   Combining all of the above, we obtain
    \begin{align} \label{bN asymp}
        b_N = \frac{1}{2\alpha\sqrt{\pi}} (1 - \erf(\alpha)) N + \Big( \frac{1-\alpha^2}{4 \pi} e^{-\alpha^2}+ \frac{\alpha}{8\sqrt{\pi}} (1 - \erf(\alpha)) \Big) + O(N^{-1}).
    \end{align}
    Then the lemma follows from \eqref{aN asymp} and \eqref{bN asymp}. 
\end{proof}

Before we compute the finite size corrections of $\bfR_N^1(x)$ and $\bfR_N^2(x)$, we need some preparation.
In the sequel, we shall use the notations 
\begin{align}
    \widetilde{\theta}_m(x) &:= \theta_m(x) - \tfrac{\alpha^2 x}{2} \phi'(x), \label{Tilde theta}
    \\
    \Theta(x) &:= \widetilde{\theta}_{-1}(x) + \widetilde{\theta}_{-2}(x), \label{Theta}
\end{align}
where $\phi$ and $\theta_m$ are defined in \eqref{phi} and \eqref{theta_m}.
The following lemma is a reformulation of Lemma~\ref{Lem_Plancherel-Rotach osc} that is directly applicable for our purpose.
\begin{lemma}
\label{Lem_Plancherel oscillatory regime, tau fixed}
    Let $\tau = 1 - \alpha^2/N$ with fixed $\alpha \in [0, \infty)$.
    Then as $N \to \infty$, we have
    \begin{align} \label{eq. Plancherel oscillatory regime (tau fixed)}
    \begin{split}
    e^{-\frac{Nx^2}{2(1+\tau)}} H_{N+m}(\sqrt{\tfrac{N}{2\tau}}x) 
    &= \frac{ e^{\frac{\alpha^2}{8}x^2} }{ \sqrt{\pi} ( 1 - \frac{x^2}{4} )^{\frac{1}{4}} }  e^{\frac{N}{2}} (N+m)! 2^{\frac{N+m}{2}} N^{-\frac{N+m+1}{2}}
    \\
    &\qquad \times \Big( k_{(0)}^{(m)}(x) + k_{(1)}^{(m)}(x) N^{-1} + O(N^{-2}) \Big)
    \end{split}
    \end{align}
    for any $x\in(-2,2)$, where
    \begin{align} \label{k_(j)^(m)}
    \begin{split}
    k_{(0)}^{(m)}(x) &:= \cos( N \phi(\tfrac{x}{2}) - \widetilde{\theta}_m(\tfrac{x}{2}) ), 
    \\
    k_{(1)}^{(m)}(x) &:= C_m(\tfrac{x}{2}) \cos( N \phi(\tfrac{x}{2}) - \widetilde{\theta}_m(\tfrac{x}{2}) ) + D_m(\tfrac{x}{2}) \sin( N \phi(\tfrac{x}{2}) - \widetilde{\theta}_m(\tfrac{x}{2}) ).
    \end{split}
    \end{align}
    Here, $\phi$ and $\widetilde{\theta}_m$ are given in \eqref{phi} and \eqref{Tilde theta}, and
    \begin{align}
    \begin{split}\label{Coeff Cm Dm}
    C_m(x) &:= A_m(x) + \frac{ (3 \alpha^2 - 1 ) \alpha^2 x^2 - 3 \alpha^4 x^4}{4(1-x^2)}, 
    \\
    D_m(x) &:= B_m(x) - \frac{(3 \alpha^2 + 2m+1)\alpha^2 x - 4 \alpha^4 x^3}{4\sqrt{1-x^2}},
    \end{split}
    \end{align}
    where $A_m$ and $B_m$ are given in \eqref{coeff Am Bm}.
\end{lemma}

We observe from \eqref{eq. RN2 integral representation}, \eqref{eq. RN1 integral representation} and \eqref{eq. Plancherel oscillatory regime (tau fixed)} that the asymptotic behaviours of $\bfR_N^j(x)$ involve certain oscillatory integrals.
To analyse such integrals, we shall utilize the following lemma. 

\begin{lemma} \label{Lem_oscillatory integral}
    Let $f$ and $\psi$ be $\mathcal{C}^2$-functions on a neighborhood of an interval $[a, b]$.
    Suppose $\psi$ has no critical point in $[a, b]$.
    Then as $N \to \infty$, we have
    \begin{equation} \label{eq. asymp. of oscillatory integral}
        \int_{a}^{b} f(u) e^{iN\psi(u)} du = i \Big( \frac{f(a)}{\psi'(a)} e^{iN\psi(a)} - \frac{f(b)}{\psi'(b)} e^{iN\psi(b)} \Big) N^{-1}  + O(N^{-2}).
    \end{equation}
\end{lemma}
\begin{proof}
    Since $\psi'$ has no zero in $[a,b]$, it follows from integration by parts that
    \begin{equation*}
        \int_{a}^{b} f(u) e^{iN\psi(u)} \,du
        = i \bigg[ \Big( \frac{f(a)}{\psi'(a)} e^{iN\psi(a)} - \frac{f(b)}{\psi'(b)} e^{iN\psi(b)} \Big)  + \int_{a}^{b} \Big(\frac{f(u)}{\psi'(u)}\Big)' e^{iN\psi(u)} \,du \bigg]  N^{-1}.
    \end{equation*}
    Again, by integration by parts, we have
    \begin{equation*}
        \Big| \int_{a}^{b} \Big(\frac{f(u)}{\psi'(u)}\Big)' e^{iN\psi(u)} \,du \Big|
        \leq \bigg( 2\Big\Vert \frac{1}{\psi'} \Big(\frac{f}{\psi'}\Big)' \Big\Vert_{\infty} + x \Big\Vert \Big( \frac{1}{\psi'} \Big(\frac{f}{\psi'}\Big)' \Big)'\Big\Vert_{\infty} \bigg) N^{-1} = O(N^{-1}).
    \end{equation*}
    This finishes the proof.
\end{proof}

Now we are ready to compute the finite size corrections for $\bfR_N^j(x)$'s.
We begin with $\bfR_N^1(x)$.
 
\begin{lemma} \label{Lem_RN1(x) wH}
    Let $\tau = 1 - \alpha^2/N$ with fixed $\alpha \in [0,\infty)$.
    Then as $N \to \infty$, we have
    \begin{equation}
    \bfR_N^1(x) = \bfR_{(0)}^1(x) N + \bfR_{(1)}^1(x) + O(N^{-1}),
    \end{equation}
    for any $x \in (-2,2)$, where
    \begin{align}
    \bfR_{(0)}^1(x) &= \frac{1}{2 \sqrt{\pi} \alpha} \erf(\tfrac{\alpha}{2} \sqrt{4-x^2}),
    \\
    \begin{split}
    \bfR_{(1)}^1(x) &= \frac{\alpha}{8 \sqrt{\pi} } \erf(\tfrac{ \alpha }{2} \sqrt{4-x^2}) - \frac{ 3 \alpha^2 x^2-4 \alpha^2 + 8 }{8 \pi \sqrt{4-x^2}} e^{\frac{\alpha ^2}{4}(x^2-4)}
    \\
    &\quad -\frac{ e^{\frac{\alpha^2}{4}x^2} }{ \pi (4-x^2) } \sin( 2N\phi(\tfrac{x}{2}) - \Theta(\tfrac{x}{2}) ) + \frac{1}{4 \pi} e^{-\alpha ^2}.
    \end{split}
    \end{align}
    Here, $\phi$ and $\Theta$ are given in \eqref{phi} and \eqref{Theta}.
\end{lemma}
\begin{proof}
    Note that by \eqref{eq. RN1 integral representation}, we have
    \begin{align*}
    \begin{split}
    \bfR_N^1(x) - \bfR_N^1(0) &= - \sqrt{\frac{2}{\pi}} e^{-\alpha^2} 2^{-N+\frac{1}{2}} \Big( 1 + \tfrac{4\alpha^2 - \alpha^4}{2} N^{-1} + O(N^{-2}) \Big) \frac{N}{(N-2)!}
    \\
    & \quad \times \int_0^x e^{-\frac{N}{1+\tau}u^2} H_{N-2}(\sqrt{\tfrac{N}{2\tau}}u) H_{N-1}(\sqrt{\tfrac{N}{2\tau}}u) \,du.
    \end{split}
    \end{align*}
    Using \eqref{eq. Plancherel oscillatory regime (tau fixed)} and Stirling's formula
    \begin{equation}\label{eq. Stirling}
        N! = \sqrt{2\pi N} \Big( \frac{N}{e} \Big)^{N} \Big( 1 + \frac{1}{12} N^{-1} + O(N^{-2}) \Big),
    \end{equation}
    we obtain
    \begin{align}
    \begin{split} \label{RN1(x)-RN1(0)}
    \bfR_N^1(x) - \bfR_N^1(0)
    &= - N e^{-\alpha^2} \Big( 1 + (\tfrac{4\alpha^2 - \alpha^4}{2} + \tfrac{1}{12}) N^{-1} + O(N^{-2}) \Big)
    \\
    &\quad \times \Big( K_{(0)}(x) + K_{(1)}(x) N^{-1} + O(N^{-2}) \Big),
    \end{split}
    \end{align}
    where
    \begin{align*}
        K_{(0)} &:= \int_0^x \frac{ e^{\frac{\alpha^2}{4}u^2} }{ \pi ( 1 - \frac{u^2}{4} )^{\frac{1}{2}} }   k_{(0)}^{(-2)}(u) k_{(0)}^{(-1)}(u) \,du,
        \\
        K_{(1)} &:= \int_0^x \frac{ e^{\frac{\alpha^2}{4}u^2} }{ \pi ( 1 - \frac{u^2}{4} )^{\frac{1}{2}} }   \Big( k_{(1)}^{(-2)}(u) k_{(0)}^{(-1)}(u) + k_{(0)}^{(-2)}(u) k_{(1)}^{(-1)}(u) \Big) \,du.
    \end{align*}
    Here, $k_{(j)}^{(m)}$'s are defined in \eqref{k_(j)^(m)}.

    We first evaluate $K_{(0)}(x)$.
    By the definition \eqref{k_(j)^(m)}, we have
    \begin{align*}
    k_{(0)}^{(-2)}(u) k_{(0)}^{(-1)}(u) = \frac{u}{4} + \frac{1}{2} \cos( 2N \phi(\tfrac{u}{2}) - \Theta(\tfrac{u}{2}) ).
    \end{align*}
    Therefore, we have
    \begin{align}
    \begin{split}
    \label{K(0) second expression}
    K_{(0)}(x)  &  = \int_0^{x} \frac{ u \, e^{\frac{\alpha^2}{4}u^2} }{ 4\pi( 1 - \frac{u^2}{4} )^{\frac{1}{2}} } \, du + \re\Big[ \int_0^x f(\tfrac{u}{2}) e^{i 2N \phi(\tfrac{u}{2})} \,du \Big]
    \\
    &= \frac{e^{\alpha^2} }{2 \sqrt{\pi} \alpha} \Big( \erf(\alpha )-\erf(\tfrac{\alpha}{2} \sqrt{4-x^2})\Big) + \re\Big[ \int_0^x f(\tfrac{u}{2}) e^{i 2N \phi(\tfrac{u}{2})} \,du \Big],
    \end{split}
    \end{align}
    where
    \begin{equation*}
    f(u) := \frac{ 1 }{2 \pi (1-u^2)^\frac{1}{2}} \exp\Big( \alpha^2 u^2 - i \Theta(u) \Big).
    \end{equation*}
    We now apply \eqref{eq. asymp. of oscillatory integral} to obtain
    \begin{align*}
    \begin{split}
    \re\Big[ \int_0^x f(\tfrac{u}{2}) e^{i 2N \phi(\tfrac{u}{2})} \,du \Big] &= \im\Big[ \Big( \frac{f(\frac{x}{2})}{\phi'(\frac{x}{2})} e^{i2N\phi(\frac{x}{2})} - \frac{f(0)}{\phi'(0)} e^{i2N\phi(0)} \Big)  \Big] N^{-1} + O(N^{-2})
    \\
    &= \frac{ e^{\frac{\alpha^2}{4}x^2} }{ \pi (4-x^2) } \sin( 2N\phi(\tfrac{x}{2}) - \Theta(\tfrac{x}{2}) ) N^{-1} + O(N^{-2}).
    \end{split}
    \end{align*}
    Putting them together, we have
    \begin{equation} \label{K(0) final expression}
      K_{(0)}(x) = \frac{e^{\alpha^2} }{2 \sqrt{\pi} \alpha} \Big( \erf(\alpha )-\erf(\tfrac{\alpha}{2} \sqrt{4-x^2})\Big) + \frac{ e^{\frac{\alpha^2}{4}x^2} }{ \pi (4-x^2) } \sin( 2N\phi(\tfrac{x}{2}) - \Theta(\tfrac{x}{2}) ) N^{-1} + O(N^{-2}).
    \end{equation}
    
    We now turn into evaluation of $K_{(1)}(x)$.
    By the definition \eqref{k_(j)^(m)}, we have
    \begin{align*}
        k_{(1)}^{(-2)}(u) k_{(0)}^{(-1)}(u)
        &= \frac{C_{-2}(\tfrac{u}{2})}{2} \Big[ \cos( 2N \phi(\tfrac{u}{2}) - \Theta(\tfrac{u}{2}) ) + \frac{u}{2} \Big]  
        \\
        &\quad + \frac{D_{-2}(\tfrac{u}{2})}{2} \Big[ \sin( 2N \phi(\tfrac{u}{2}) - \Theta(\tfrac{u}{2}) ) + \Big(1-\frac{u^2}{4}\Big)^\frac{1}{2} \Big],
    \end{align*}
    and
    \begin{align*}
        k_{(0)}^{(-2)}(u) k_{(1)}^{(-1)}(u) 
        &= \frac{C_{-1}(\tfrac{u}{2})}{2} \Big[ \cos( 2N \phi(\tfrac{u}{2}) - \Theta(\tfrac{u}{2})  ) + \frac{u}{2} \Big]
        \\
        &\quad  + \frac{D_{-1}(\tfrac{u}{2})}{2} \Big[ \sin( 2N \phi(\tfrac{u}{2}) - \Theta(\tfrac{u}{2}) ) - \Big(1-\frac{u^2}{4}\Big)^\frac{1}{2} \Big].
    \end{align*}
    Then, by using \eqref{Coeff Cm Dm} and \eqref{eq. asymp. of oscillatory integral}, after some computations, we have 
    \begin{align} \label{K(1) final expression}
    \begin{split}
    K_{(1)}(x) &=\frac{6 \alpha ^4- 21 \alpha ^2 - 1}{24 \sqrt{\pi}  \alpha } e^{\alpha ^2} \Big( \erf(\alpha )-\erf(\tfrac{\alpha}{2} \sqrt{4-x^2}) \Big) 
    \\
    & \quad + \frac{ 3 \alpha^2 x^2-4 \alpha^2 + 8 }{8 \pi \sqrt{4-x^2}}e^{\frac{\alpha ^2 x^2}{4}}+ \frac{\alpha^2-2}{4 \pi} + O(N^{-1}).
    \end{split}
    \end{align}
    Substituting \eqref{K(0) final expression} and \eqref{K(1) final expression} into \eqref{RN1(x)-RN1(0)}, the conclusion follows.
\end{proof}

Next, we compute the asymptotic behaviour of $\bfR_N^2(x)$.
The proof is analogous to the previous proof, but much simpler.

\begin{lemma} \label{Lem_RN2(x) wH}
    Let $\tau = 1 - \alpha^2/N$ with fixed $\alpha \in [0,\infty)$.
    Then as $N \to \infty$, we have
    \begin{equation}
        \bfR_N^{2}(x) = \frac{e^{\frac{\alpha^2}{4}(x^2-4)}}{\pi (4-x^2)} 
        \bigg[ \frac{\sqrt{4-x^2}}{2} + \sin( 2N \phi(\tfrac{x}{2}) - \Theta(\tfrac{x}{2}) ) \bigg] + O(N^{-1}),
    \end{equation}
    for any $x \in (-2,2)$. Here, $\phi$ and $\Theta$ are given in \eqref{phi} and \eqref{Theta}.
\end{lemma}

\begin{proof}
    By using \eqref{eq. RN2 integral representation} and \eqref{eq. Plancherel oscillatory regime (tau fixed)}, we have
    \begin{align*}
        \begin{split}
        \bfR_N^2(x) &= \frac{N}{\pi} \frac{e^{-\alpha^2+\tfrac{\alpha^2}{8}x^2}}{(4-x^2)^{\frac{1}{4}}} \cos( N \phi(\tfrac{x}{2}) - \widetilde{\theta}_{-1}(\tfrac{x}{2}) ) \re\Big[ \int_0^x f(\tfrac{u}{2}) e^{iN \phi(\tfrac{u}{2})}\,du\Big] \cdot \Big( 1+O(N^{-1}) \Big),
        \end{split}
    \end{align*}
    where
    \begin{equation*}
        f(u) := \frac{1}{\sqrt{2} (1-u^2)^{\frac{1}{4}}} \exp\Big( \tfrac{\alpha^2}{2} u^2 - i \, \widetilde{\theta}_{-2}(u) \Big).
    \end{equation*}
    We observe that the derivative $\frac{d}{du} \phi(\tfrac{u}{2}) = \sqrt{1-u^2/4}$ does not vanish in $(-2, 2)$. 
    Therefore it follows from \eqref{eq. asymp. of oscillatory integral} that
    \begin{align*}
        \re\Big[ \int_0^x f(u) e^{iN\phi(\tfrac{u}{2})} \,du\Big] &= \im\Big[ 2 \Big( \frac{f(x)}{\phi'(\tfrac{x}{2})} e^{iN\phi(\tfrac{x}{2})} - \frac{f(0)}{\phi'(0)} e^{iN\phi(0)} \Big) N^{-1}  + O(N^{-2}) \Big]
        \\
        &= \frac{2 e^{\frac{\alpha^2}{8} x^2}}{(4-x^2)^{\frac34}}
        \sin( N \phi(\tfrac{x}{2}) - \widetilde{\theta}_{-2}(\tfrac{x}{2}) ) N^{-1} + O(N^{-2}).
    \end{align*}
    This completes the proof.
\end{proof}

\section{Edge scaling densities} \label{sec_edge}

In this section, we prove Theorem~\ref{Thm_Edge density asym}.
According to the rescalings given in \eqref{rescaled RN sH} and \eqref{rescaled RN wH}, let us write
\begin{equation} \label{rescaled RNj's}
    R_N^j(\xi) = \frac{c}{N^r} \bfR_N^j(1+\tau + \frac{c}{N^r} \xi), \qquad (j = 1, 2),
\end{equation}
where $c = \sqrt{1-\tau^2}$ and $r = 1/2$ (\textit{resp.}, $c=1$ and $r = 2/3 $) and $\bfR_N^j$'s are given in \eqref{eq. RN1 discrete summation} and \eqref{eq. RN2 integral representation}.
For the later purpose, let 
\begin{equation} \label{rescaled RNj = SN TNj}
    R_N^j(\xi) = S_N \cdot T_N^j(\xi), \qquad (j = 1, 2),
\end{equation}
where
\begin{equation} \label{rescaled SN}
    S_N := \frac{(\tau/2)^{N-\frac{3}{2}}}{1+\tau} \frac{N}{(N-2)!} \frac{c}{N^r}
\end{equation}
and
\begin{align} \label{rescaled TN1}
    T_N^1(\xi) &:=  \sqrt{\frac{2}{\pi}} \int_\xi^\infty e^{-\frac{N}{1+\tau}u^2} H_{N-2}(\sqrt{\tfrac{N}{2\tau}}u) H_{N-1}(\sqrt{\tfrac{N}{2\tau}}u) \,du,
    \\
    \begin{split} \label{rescaled TN2 (1)}
    T_N^2(\xi) &:= \frac{1}{\sqrt{2\pi}} e^{-\frac{N}{2(1+\tau)}x^2} H_{N-1}(\sqrt{\tfrac{N}{2\tau}}x) \int_0^x e^{-\frac{N}{2(1+\tau)}u^2}H_{N-2}(\sqrt{\tfrac{N}{2\tau}}u) \,du.
    \end{split}
\end{align}

We shall use an integration formula for the Hermite polynomials which can be found in \cite[Eq.(7.376)]{GR14}: for an even number $N$ and $a>0$, we have
\begin{equation*}
    \int_0^\infty e^{ -a t^2 } H_N(t) dt = \frac{\sqrt{\pi}}{2} \frac{N!}{(N/2)!} a^{-\frac{N+1}{2}}(1-a)^{\frac{N}{2}}.
\end{equation*}
Then we have 
\begin{equation} \label{rescaled TN2 (2)}
    T_N^2(\xi) = \frac{1}{\sqrt{2\pi}} e^{-\frac{N}{2(1+\tau)}x^2} H_{N-1}(\sqrt{\tfrac{N}{2\tau}}x) \bigg( \sqrt{\frac{\pi(1+\tau)}{2N}} \frac{(N-2)!}{(N/2-1)!} \tau^{1-\frac{N}{2}} - \widetilde{T}_N^2(\xi) \bigg),
\end{equation}
where
\begin{equation} \label{Tilde TN2}
    \widetilde{T}_N^2(\xi) := \int_x^\infty e^{-\frac{N}{2(1+\tau)}u^2}H_{N-2}(\sqrt{\tfrac{N}{2\tau}}u) \,du.
\end{equation}

The outline of this section is as follows:
\begin{itemize}
    \item the Plancherel-Rotach formula for our specific purpose is formulated in Lemma~\ref{Lem_P-R exp tau fixed} (\textit{resp.}, Lemma~\ref{Lem_P-R tran tau varies});
    \item we compute the asymptotic behaviour of $R_N^1(\xi)$ by analysing $S_N$ and $T_N^1(\xi)$ in Lemma~\ref{Lem_RN1(xi) sH edge} (\textit{resp.}, Lemma~\ref{Lem_RN1(xi) wH edge});
     \smallskip 
    \item analogous to the previous step, we derive the asymptotic behaviour of $R_N^2(\xi)$ by analysing $\widetilde{T}_N^2(\xi)$ in Lemma~\ref{Lem_RN2(xi) sH edge} (\textit{resp.}, Lemma~\ref{Lem_RN2(xi) wH edge}).
\end{itemize}

\subsection{Strong non-Hermiticity}
In this subsection, we prove Theorem~\ref{Thm_Edge density asym} (i).
Throughout this subsection, we write
\begin{equation} \label{x and xi}
    x = 1 + \tau + \sqrt{\frac{1-\tau^2}{N}} \xi.
\end{equation}
We frequently use the following reformulation of Lemma~\ref{Lem_Plancherel-Rotach exp}.
\begin{lemma}
\label{Lem_P-R exp tau fixed}
    Let $\tau \in (0, 1)$ be fixed. Then as $N \to \infty$, we have
    \begin{equation} \label{eq. Plancherel exponential regime (tau fixed)}
        e^{-\frac{N}{2(1+\tau)} x^2} H_{N+m}(\sqrt{\tfrac{N}{2\tau}}x) = \frac{e^{\frac{m}{2}}}{\sqrt{1-\tau}} \Big( \frac{2N}{e\tau} \Big)^{\frac{N+m}{2}} e^{-\xi^2} \Big(1 + h_{(1)}^{\rm exp}(\xi) N^{-\frac{1}{2}}  + O(N^{-1})\Big)
    \end{equation}
    for $\xi \in \R$, where
    \begin{equation}
        h_{(1)}^{\rm exp}(\xi) := \frac{1}{3} \Big(\frac{1+\tau}{1-\tau}\Big)^\frac{3}{2}\xi^3 + \frac{(1-\tau)m-\tau}{1-\tau}\Big(\frac{1+\tau}{1-\tau}\Big)^\frac{1}{2}\xi.
    \end{equation}
\end{lemma}
\begin{rem} \label{Rem_Error bound for P-R exp}
    The error bound of $O(N^{-1})$ in \eqref{eq. Plancherel exponential regime (tau fixed)} was obtained in \cite[Theorem 2.1]{SNWW23}.
    Furthermore, it was shown that if $\xi$ is bounded away from below, then the error term is uniformly bounded by a polynomial of $\xi$.
    Then, integrating \eqref{eq. Plancherel exponential regime (tau fixed)} with respect to $\xi$ from $+\infty$, we observe that the error term remains $O(N^{-1})$ after the integration.
\end{rem}

We first derive the asymptotic behaviours of $R_N^1.$

\begin{lemma} \label{Lem_RN1(xi) sH edge}
    Let $\tau \in (0,1)$ be fixed. Then as $N \to \infty$, we have
    \begin{equation} \label{rescaled RN1 sH}
        R_N^1(\xi) = R_{(0)}^1(\xi) + R_{(1)}^1(\xi) N^{-\frac{1}{2}} + O(N^{-1}),
    \end{equation}
    for $\xi \in \R$, where
    \begin{align}
        R_{(0)}^1(\xi) &:= \frac{ 1 }{2 \sqrt{ 2 \pi }} \Big(1-\erf(\sqrt{2} \xi ) \Big),
        \\
        R_{(1)}^1(\xi) &:= \frac{\sqrt{1-\tau ^2}}{6 \pi (1-\tau )^2} e^{-2 \xi^2} \Big(( 1 +\tau ) \xi^2 + 2 \tau -4 \Big).
    \end{align}
\end{lemma}
\begin{proof}
    By \eqref{rescaled RNj = SN TNj}, it is enough to compute the asymptotic behaviours of $S_N$ and $T_N^1(\xi)$.
    Note that by Stirling's formula \eqref{eq. Stirling}, we have
    \begin{equation} \label{asymp of SN}
        S_N = \Big( \frac{1-\tau}{2\pi(1+\tau) } \Big)^\frac{1}{2} e^{\frac{3}{2}} N^{\frac{1}{2}} \Big(\frac{e\tau}{2N}\Big)^{N-\frac{3}{2}} \Big(1 + O(N^{-1}) \Big).
    \end{equation}
    On the other hand, by \eqref{eq. Plancherel exponential regime (tau fixed)} and Remark~\ref{Rem_Error bound for P-R exp}, we have 
    \begin{align*}
    \begin{split}
        T_N^1(\xi) &= \Big( \frac{2(1+\tau)}{\pi(1-\tau)} \Big)^\frac{1}{2} e^{-\frac{3}{2}} N^{-\frac{1}{2}} \Big(\frac{2N}{e\tau}\Big)^{N-\frac{3}{2}}
        \\
        & \quad \times \int_\xi^\infty e^{-2t^2} \bigg(1 + \Big(\frac{2}{3} \Big( \frac{1+\tau}{1-\tau} \Big)^\frac{3}{2}t^3- \frac{3-\tau}{1-\tau} \Big(\frac{1+\tau}{1-\tau} \Big)^\frac{1}{2}t \Big) N^{-\frac{1}{2}} + O(N^{-1})\bigg)  \,dt.
    \end{split}
    \end{align*}
    Using the Gaussian integral
    \begin{equation} \label{Gaussian integral}
        \int_x^\infty t^3e^{-a^2t^2} dt = \frac{a^2x^2+1}{2a^4} e^{-a^2x^2}, \qquad (a >0),
    \end{equation}
    we obtain
    \begin{align*}
        T_N^1(\xi) = \Big( \frac{2 \pi (1+\tau)}{1-\tau} \Big)^\frac{1}{2} e^{-\frac{3}{2}} N^{-\frac{1}{2}} \Big(\frac{2N}{e\tau} \Big)^{N-\frac{3}{2}}
        \Big( R_{(0)}^1(\xi) + R_{(1)}^1(\xi) N^{-\frac{1}{2}} + O(N^{-1}) \Big).
    \end{align*}
    These implies the assertion of the lemma when substituted in \eqref{rescaled RNj = SN TNj}.
\end{proof}

Next, we compute the asymptotic behaviours of $R_N^2.$

\begin{lemma} \label{Lem_RN2(xi) sH edge}
    Let $\tau \in (0,1)$ be fixed. Then as $N \to \infty$, we have
    \begin{equation}
        R_N^2(\xi) = R_{(0)}^2(\xi) + R_{(1)}^2(\xi) N^{-\frac{1}{2}} + O(N^{-1}),
    \end{equation}
    for $\xi \in \R$, where
    \begin{align}
        R_{(0)}^2(\xi) &:= \frac{1}{4\sqrt{\pi} } e^{-\xi^2} \Big( 1 + \erf(\xi) \Big),
        \\
        R_{(1)}^2(\xi) &:= \frac{ \sqrt{1-\tau^2} }{ 12 \pi (1-\tau)^2 } e^{-2\xi^2} \bigg( \Big( ( 1 + \tau) \xi^3  - 3 \xi \Big) \sqrt{\pi} e^{\xi^2}  \Big(1+\erf(\xi ) \Big) - (1+\tau)\xi^2 - 4 \tau + 5 \bigg).
    \end{align}
\end{lemma}
\begin{proof}
    We begin with \eqref{rescaled TN2 (2)}.
    Making use of Stirling's formula \eqref{eq. Stirling}, we have
    \begin{equation*}
        \sqrt{\frac{\pi(1+\tau)}{2N}} \frac{(N-2)!}{(N/2-1)!} \tau^{-\frac{N}{2}+1} = \frac{\tau \sqrt{\pi(1+\tau)}}{2} \Big(\frac{2N}{e\tau}\Big)^\frac{N}{2} N^{-\frac{3}{2}} \Big( 1 + O(N^{-1}) \Big).
    \end{equation*}
    By using \eqref{eq. Plancherel exponential regime (tau fixed)} and Remark~\ref{Rem_Error bound for P-R exp}, we obtain
    \begin{align*}
    \begin{split}
        \widetilde{T}_N^2(\xi) &= \frac{ \tau \sqrt{1+\tau} }{2} \Big(\frac{2N}{e\tau}\Big)^{\frac{N}{2}} N^{-\frac{3}{2}}
        \\
        &\quad \times \int_\xi^\infty e^{-t^2} \bigg( 1 + \Big(\frac{1}{3} \Big( \frac{1+\tau}{1-\tau} \Big)^\frac{3}{2}t^3 + \frac{\tau-2}{1-\tau} \Big( \frac{1+\tau}{1-\tau} \Big)^\frac{1}{2}t\Big) N^{-\frac{1}{2}} + O(N^{-1}) \bigg) \,dt.
    \end{split}
    \end{align*}
    Then, by the Gaussian integral \eqref{Gaussian integral}, we have
    \begin{align*}
    \begin{split}
        \widetilde{T}_N^2(\xi) &= \frac{\tau\sqrt{\pi(1+\tau)}}{2} \Big(\frac{2N}{e\tau}\Big)^{\frac{N}{2}} N^{-\frac{3}{2}} 
        \\
        &\quad \times \bigg( \frac{1}{2}\Big(1 - \erf(\xi)\Big) + \frac{\sqrt{1+\tau} }{6 \sqrt{\pi} (1-\tau )^{\frac{3}{2}}} \Big( (1 + \tau) \xi^2 +4 \tau -5\Big) e^{-\xi^2} N^{-\frac{1}{2}} + O(N^{-1}) \bigg).
    \end{split}
    \end{align*}
    Putting them together in \eqref{rescaled TN2 (2)}, we have
    \begin{equation}
        T_N^2(\xi) = \Big( \frac{2\pi(1+\tau)}{1-\tau} \Big)^{\frac{1}{2}} e^{-\frac{3}{2}} N^{-\frac{1}{2}} \Big(\frac{2N}{e\tau}\Big)^{N-\frac{3}{2}} e^{-\xi^2} \Big( R_{(0)}^2(\xi) + R_{(1)}^2(\xi) N^{-\frac{1}{2}} + O(N^{-1}) \Big).
    \end{equation}
    Combining this with \eqref{asymp of SN}, the lemma follows. 
\end{proof}

\subsection{Weak non-Hermiticity}
In this subsection, we prove Theorem~\ref{Thm_Edge density asym} (ii).
Throughout this subsection, let us write
\begin{equation}
    x = 1 + \tau + \frac{\xi}{N^{2/3}}.
\end{equation}

For the later purpose, we state a generalised version of Lemma~\ref{Lem_Plancherel-Rotach crit}. See e.g. \cite[Eqs.(3.3) and (3.5)]{Sk59}.
\begin{lemma} \label{Lem_P-R tran tau varies}
    Let $\tau = 1 - \alpha^2/N^{\frac{1}{3}}$ with fixed $\alpha \in [0,\infty)$. Then as $N \to \infty$, we have
    \begin{align}\label{eq. Plancherel transition regime (tau varies)}
    \begin{split}
    e^{-\frac{Nx^2}{2(1+\tau)}} H_{N+m}(\sqrt{\tfrac{N}{2\tau}}x) &= \exp\Big( \tfrac{\alpha^2}{2} N^{\frac23} + \tfrac{\alpha^4}{4} N^{\frac13} + \tfrac{\alpha^6}{6} \Big) (2N)^{\frac{m}{2}}\pi^{\frac14}2^{\frac{N}{2}+\frac14}(N!)^{\frac12}N^{-\frac{1}{12}}
    \\
    & \quad \times \Big( \Airy_\alpha(\xi) + \big( \mathcal{A}_m(\xi) \Airy_\alpha(\xi) + \mathcal{B}_m(\xi) \Airy_\alpha'(\xi) \big) N^{-\frac13} + O(N^{-\frac23 + \epsilon}) \Big),
    \end{split}
    \end{align}
    uniformly for $\xi \in [-N^\delta, N^\delta]$, where $\delta \in [0,2/3)$, $\epsilon > 2 \delta$ and
    \begin{equation}
        \mathcal{A}_m(t) := \frac{2\alpha^4t+\alpha^8+(4m+2)\alpha^2}{8},
        \qquad
        \mathcal{B}_m(t) := \frac{2\alpha^2 t + \alpha^6 -4m-2}{4}.
    \end{equation}
\end{lemma}

As before, we compute the asymptotic behaviours of $R_N^1.$

\begin{lemma} \label{Lem_RN1(xi) wH edge}
    Let $\tau = 1 - \alpha^2 / N^\frac{1}{3}$. Then as $N \to \infty$, we have
    \begin{equation}
        R_N^1(\xi)= R_{(0)}^1(\xi) + R_{(1)}^1(\xi) N^{-\frac{1}{3}} + O(N^{-\frac{2}{3}+\epsilon}),
    \end{equation}
    for any $\epsilon>0$, where
    \begin{align} 
        R_{(0)}^1(\xi) &= \int_\xi^\infty \Airy_\alpha(t)^2\,dt,
        \\
        R_{(1)}^2(\xi) &= -\frac{2\alpha^2\xi+\alpha^6+4}{4} \Airy_\alpha(\xi)^2 + \int_\xi^\infty \frac{\alpha^4t+ \alpha^2}{2} \Airy_\alpha(t)^2 \,dt.
    \end{align}
\end{lemma}
\begin{proof}
    Note that as $N \to \infty,$
    \begin{equation} \label{asymp tau^N wH}
        \tau^N
        = \Big( 1-\frac{\alpha^8}{4}N^{-\frac{1}{3}}+O(N^{-\frac{2}{3}}) \Big) \exp\Big( -\alpha^2N^\frac{2}{3}-\tfrac{\alpha^4}{2}N^\frac{1}{3}-\tfrac{\alpha^6}{3} \Big).
    \end{equation}
    Then it follows from \eqref{rescaled SN} and \eqref{eq. Stirling} that
    \begin{equation} \label{asymp of SN wH}
        S_N = \frac{1}{\sqrt{\pi}} \Big(\frac{e}{2N}\Big)^N N^{\frac{11}{6}} \Big( 1+(2\alpha^2-\tfrac{\alpha^8}{4})N^{-\frac{1}{3}}+O(N^{-\frac{2}{3}}) \Big) \exp\Big( -\alpha^2N^\frac{2}{3}-\tfrac{\alpha^4}{2}N^\frac{1}{3}-\tfrac{\alpha^6}{3} \Big) .
    \end{equation}
    
    Next, we analyse the asymptotic behaviour of $T_N^1(\xi)$.
    Let us write
    \begin{equation} \label{r_N}
        r_N := 1 + \tau + \frac{N^\delta}{N^{2/3}} ,
    \end{equation}
    for a sufficiently small $\delta>0$.
    Then we divide
    \begin{equation*}
        T_N^1(\xi) = T_N^{1,1}(\xi) + T_N^{1,2}(\xi),
    \end{equation*}
    where
    \begin{align*}
        T_N^{1,1}(\xi) &:= \sqrt{\frac{2}{\pi}} \int_x^{r_N} e^{-\frac{N}{1+\tau}u^2} H_{N-2}(\sqrt{\tfrac{N}{2\tau}}u) H_{N-1}(\sqrt{\tfrac{N}{2\tau}}u) \, du,
        \\
        T_N^{1,2}(\xi) &:= \sqrt{\frac{2}{\pi}} \int_{r_N}^\infty e^{-\frac{N}{1+\tau}u^2} H_{N-2}(\sqrt{\tfrac{N}{2\tau}}u) H_{N-1}(\sqrt{\tfrac{N}{2\tau}}u) \, du.
    \end{align*}
    
    Using \eqref{eq. Plancherel transition regime (tau varies)} and the Airy differential equation $\Airy''(x) = x \Airy(x)$, we deduce 
    \begin{align*} 
    \begin{split}
        T_N^{1,1}(\xi) &= \exp\Big( \alpha^2 N^{\frac{2}{3}} + \tfrac{\alpha^4}{2} N^{\frac{1}{3}} + \tfrac{\alpha^6}{3} \Big) (2N)^{-\frac{3}{2}} 2^{N+1} (N!) N^{-5/6} \\
        & \quad \times \bigg( \int_\xi^{N^\delta} \Big( \Airy_\alpha(t)^2 
        + \big( a_1(t) \Airy_\alpha(t)^2 + a_2(t) \Airy_\alpha(t)\Airy_\alpha'(t) \big) N^{-\frac{1}{3}} \Big) \,dt + O(N^{-\frac{2}{3} + 3 \delta} ) \bigg),
    \end{split}
    \end{align*}
    where
    \begin{equation*}
        a_1(t) := \frac{2\alpha^4t+\alpha^8-4\alpha^2}{4},
        \qquad
        a_2(t) := \frac{2\alpha^2t+\alpha^6+4}{2}.
    \end{equation*}
    By the asymptotic formula of the Airy function \cite[Eq.(9.7.5)]{NIST},
    \begin{equation*}
        \Airy(x) \sim \frac{1}{2\sqrt{\pi} x^{1/4}} e^{-\frac{2}{3} x^{3/2}}, \qquad (x \to +\infty),
    \end{equation*}
    we have
    \begin{equation*}
        \int_{N^\delta}^\infty \Airy_\alpha^2(t) \, dt = O(e^{-N^{\delta}}). 
    \end{equation*}
    Using this together with \eqref{eq. Stirling}, after some computations, we obtain
    \begin{align}
    \begin{split} \label{TN1 Airy wH}
        T_N^{1,1}(\xi) &= \sqrt{\pi} \Big( \frac{2N}{e} \Big)^N N^{-\frac{11}{6}} \exp\Big( \alpha^2N^\frac{2}{3} + \tfrac{\alpha^4}{2}N^\frac{1}{3} + \tfrac{\alpha^6}{3} \Big)
        \\
        &\quad \times \Big( R_{(0)}^1(\xi) + \Big(R_{(1)}^1(\xi) - (2\alpha^2-\tfrac{\alpha^8}{4}) R_{(0)}^1(\xi) \Big) N^{-\frac{1}{3}} + O(N^{-\frac{2}{3}+4\delta}) \Big).
    \end{split}
    \end{align}
    
    By \cite[Theorem 2.1]{SNWW23}, we have that for $x>1$,
    \begin{equation*}
        H_{N+m}(\sqrt{2N}x) \leq \frac{1}{(x^2-1)^{\frac{1}{4}}} e^{-\frac{N}{2}} N^{\frac{N+m-1}{2}} 2^{\frac{N+m}{2}} \sigma(x)^{N+m+\frac{3}{2}} e^{N x / \sigma(x)}. 
    \end{equation*}
    This in turn implies that for $x> r_N$, 
    \begin{equation} \label{TN 1,1 integrand}
        e^{-\frac{N}{1+\tau}x^2} H_{N-2}(\sqrt{\tfrac{N}{2\tau}}x) H_{N-1}(\sqrt{\tfrac{N}{2\tau}}x)
        \leq
        \frac{1}{(\frac{x^2}{4\tau}-1)^{\frac{1}{2}}} e^{-N} N^{N-\frac{5}{2}} 2^{N-\frac{3}{2}} \exp\Big( 2N \omega( \tfrac{x}{2\sqrt{\tau}} ) \Big),
    \end{equation}
    where
    \begin{equation*}
        \omega(x) := \log(\sigma(x)) + \frac{x}{\sigma(x)} - \frac{2\tau}{1+\tau} x^2.
    \end{equation*}
     Notice here that for a sufficiently large $N$ and $x\in[1+\tau,\infty)$, we have
    \begin{equation} \label{omega'(x)}
        \omega'(x) = - 2 \sigma(x) + \frac{2(1-\tau)}{1+\tau} x = -2 \sigma(x) + \frac{2\alpha^2}{N^{1/3}} \frac{1}{1+\tau} x < -x. 
    \end{equation}
    Furthermore, we have
    \begin{equation} \label{omega(r_N)}
        \omega(r_N) = \frac{\alpha^2}{2} N^{-\frac{1}{3}} + \frac{\alpha^4}{4} N^{-\frac{2}{3}} - N^{- 1 + \frac{3}{2}\delta} + O(N^{-1+\frac{1}{2} \delta}).
    \end{equation}
    On the other hand, by \eqref{omega'(x)} and \eqref{omega(r_N)}, we deduce
    \begin{equation} \label{omega(x) bound}
        \exp\Big( 2N\omega( \tfrac{x}{2\sqrt{\tau} }) \Big) \leq \exp\Big( \alpha^2 N^{\frac{2}{3}} + \tfrac{\alpha^4}{2} N^{\frac{1}{3}} + \tfrac{\alpha^6}{3} \Big) \exp\Big( -\tfrac{1}{2}N^{\frac{3}{2}\delta} - N x^2 \Big).
    \end{equation}
    Combining \eqref{TN 1,1 integrand} and \eqref{omega(x) bound}, we obtain 
    \begin{equation}
        T_N^{1,2}(\xi) = T_N^{1,1}(\xi) \cdot O(e^{-N^\delta}).
    \end{equation}
    Now the lemma follows from \eqref{rescaled RNj = SN TNj}, \eqref{asymp of SN wH} and \eqref{TN1 Airy wH}.
\end{proof}

It remains to show the following lemma. 

\begin{lemma} \label{Lem_RN2(xi) wH edge}
    Let $\tau = 1 - \alpha^2 / N^\frac{1}{3}$. Then as $N \to \infty$, we have
    \begin{equation}
        R_N^2(x)= R_{(0)}^2(x) + R_{(1)}^2(x) N^{-\frac{1}{3}} + O(N^{-\frac{2}{3}+\epsilon}),
    \end{equation}
    for any $\epsilon>0$, where
    \begin{align}
        R_{(0)}^2(x) &= \frac{1}{2} \Airy_\alpha(\xi) \Big( 1 - \int_\xi^\infty \Airy_\alpha(t)\,dt \Big),
        \\
        \begin{split}
        R_{(1)}^2(x) &= \Big( \frac{\alpha^4\xi+2\alpha^2}{8} \Airy_\alpha(\xi) + \frac{2\alpha^2\xi+\alpha^6+2}{8}\Airy_\alpha'(\xi) \Big) \Big( 1 - \int_\xi^\infty \Airy_\alpha(t) \,dt \Big)
        \\
        & \quad + \Airy_\alpha(\xi) \Big( \frac{2\alpha^2\xi+\alpha^6+6}{8} \Airy_\alpha(\xi) - \frac{\alpha^4}{8} \int_\xi^\infty t \Airy_\alpha(t) \,dt \Big).     
        \end{split}
    \end{align}
\end{lemma}
\begin{proof}
    It follows from \eqref{eq. Stirling} and \eqref{asymp tau^N wH} that
    \begin{align*}
        \sqrt{\frac{\pi(1+\tau)}{2N}} \frac{(N-2)!}{(N/2-1)!} \tau^{1-\tfrac{N}{2}} &=
        \sqrt{\frac{\pi}{2}} N^{-\frac{3}{2}} \Big( \frac{2N}{e} \Big)^\frac{N}{2} \exp\Big( \tfrac{\alpha^2}{2}N^\frac{2}{3} + \tfrac{\alpha^4}{4}N^\frac{1}{3}+\tfrac{\alpha^6}{6} \Big) 
        \\
        &\quad \times \Big( 1 + \frac{\alpha^8-10\alpha^2}{8}N^{-\frac{1}{3}} + O(N^{-\frac{2}{3}}) \Big).
    \end{align*}
    We decompose $\widetilde{T}_N^2$ as 
    \begin{equation*}
        \widetilde{T}_N^2(\xi) = \widetilde{T}_N^{2,1}(\xi) + \widetilde{T}_N^{2,2}(\xi),
    \end{equation*}
    where
    \begin{align*}
        \widetilde{T}_N^{2,1}(\xi) &:= \int_x^{r_N} e^{-\frac{N}{2(1+\tau)}u^2}H_{N-2}(\sqrt{\tfrac{N}{2\tau}}u) \,du,
        \\
        \widetilde{T}_N^{2,1}(\xi) &:= \int_{r_N}^\infty e^{-\frac{N}{2(1+\tau)}u^2}H_{N-2}(\sqrt{\tfrac{N}{2\tau}}u) \,du.
    \end{align*}
    Here, $r_N$ is given in \eqref{r_N}.
    Then a similar analysis, akin to the one presented in the proof of Lemma~\ref{Lem_RN1(xi) wH edge}, establishes that for a sufficiently small $\delta>0$, 
    \begin{align*}
        \begin{split}
        \widetilde{T}_N^{2,1}(\xi) &= \sqrt{\frac{\pi}{2}} N^{-\frac{3}{2}} \Big( \frac{2N}{e} \Big)^{\frac{N}{2}} \exp\Big( \tfrac{\alpha^2}{2} N^{\frac23} + \tfrac{\alpha^4}{4} N^{\frac13} + \tfrac{\alpha^6}{6} \Big) \\
        & \quad \times \bigg( \int_\xi^\infty \Big( \Airy_\alpha(t) 
        + \Big( \tfrac{2\alpha^4t + \alpha^8 - 6\alpha^2}{8} \Airy_\alpha(t) + \tfrac{2\alpha^2 t+\alpha^6+6}{4}\Airy_\alpha'(t) \Big) N^{-\frac13} \Big) \, dt + O(N^{-\frac23 + 2 \delta}) \bigg),
        \end{split}
        \\
        \widetilde{T}_N^{2,2}(\xi) &= \widetilde{T}_N^{2,1}(\xi) \cdot O(e^{-N^\delta}).
    \end{align*}
    Using \eqref{eq. Plancherel transition regime (tau varies)}, it follows from \eqref{rescaled TN2 (2)} that
    \begin{align*}
    \begin{split}
        T_N^2(\xi) &=
        \exp\Big( \tfrac{\alpha^2}{2} N^{\frac23} + \tfrac{\alpha^4}{4} N^{\frac13} + \tfrac{\alpha^6}{6} \Big) \pi^{\frac12} \Big( \frac{2N}{e} \Big)^{\frac{N}{2}} N^{-\frac{1}{3}} \\
        & \quad \times \Big( \Airy_\alpha(\xi) 
        + \Big( \frac{2\alpha^4 \xi + \alpha^8 - 2\alpha^2}{8}\Airy_\alpha(\xi) + \frac{2\alpha^2 \xi + \alpha^6 + 2}{4}\Airy_\alpha'(\xi) \Big) N^{-\frac13} + O(N^{-\frac23}) \Big)
        \\
        & \quad \times \Big( \sqrt{\frac{\pi(1+\tau)}{2N}} \frac{(N-2)!}{(N/2-1)!} \tau^{1-\tfrac{N}{2}} - \widetilde{T}_N^2(\xi) \Big).
    \end{split}
    \end{align*}
    Combining this with \eqref{rescaled RNj = SN TNj} and \eqref{asymp of SN wH}, we obtain the desired asymptotic behaviour.
\end{proof}


\begin{thebibliography}{999}


\bibitem{AB22} G.~Akemann and S.-S. Byun, \emph{The Product of $m$ real $N\times N$ Ginibre matrices: Real eigenvalues in the critical regime $m=O(N)$}, Constr. Approx. (2023). https://doi.org/10.1007/s00365-023-09628-2, arXiv:2201.07668.

\bibitem{ABES23} G.~Akemann, S.-S. Byun, M. Ebke and G. Schehr, \emph{Universality in the number variance and counting statistics of the real and symplectic Ginibre ensemble}, arXiv:2308.05519.

\bibitem{ACV18} G.~Akemann, M.~Cikovic and M.~Venker, \emph{Universality at weak and strong non-Hermiticity beyond the elliptic Ginibre ensemble}, Comm. Math. Phys. \textbf{362} (2018), 1111--1141.

\bibitem{ADM22} G. Akemann, M.~Duits and L. D.~Molag, \emph{The elliptic Ginibre ensemble: a unifying approach to local and global statistics for higher dimensions}, J. Math. Phys. \textbf{64} (2023), 023503. 

\bibitem{AK07} G.~Akemann and E.~Kanzieper, \emph{Integrable structure of Ginibre's ensemble of real random matrices and a {P}faffian integration theorem}, J. Stat. Phys. \textbf{129} (2007), 1159--1231. 

\bibitem{AP14} G. Akemann and M. J. Phillips, \emph{The interpolating Airy kernels for the $\beta = 1$ and $\beta = 4$ elliptic Ginibre ensembles}, J. Stat. Phys. \textbf{155} (2014), 421--465.

\bibitem{AK22} J. Alt and T. Kr\"{u}ger, \emph{Local elliptic law}, Bernoulli \textbf{28} (2022), no. 2, 886--909.

\bibitem{AB23} Y. Ameur and S.-S. Byun, \emph{Almost-Hermitian random matrices and bandlimited point processes}, Anal. Math. Phys. \textbf{13} (2023), 52.

\bibitem{ACCL22} Y. Ameur, C. Charlier, J. Cronvall and J. Lenells, \emph{Disk counting statistics near hard edges of random normal matrices: the multi-component regime}, arXiv:2210.13962.


\bibitem{ACM23} 
Y. Ameur, C. Charlier and P. Moreillon, \emph{Eigenvalues of truncated unitary matrices: disk counting statistics}, arXiv:2305.08976.


\bibitem{BEDPMW13} C. W. J. Beenakker, J. M. Edge, J. P. Dahlhaus, D. I. Pikulin, S. Mi and M. Wimmer, \emph{Wigner-Poisson statistics of topological transitions in a Josephson junction}, Phys. Rev. Lett. \textbf{111} (2013), 037001.

\bibitem{Be10} M. Bender, \emph{Edge scaling limits for a family of non-Hermitian random matrix ensembles}, Probab. Theory Related Fields \textbf{147} (2010), 241--271.


\bibitem{Bo16} F. Bornemann, \emph{A note on the expansion of the smallest eigenvalue distribution of the LUE at the hard edge}, Ann. Appl. Probab. \textbf{26} (2016), 1942--1946.

\bibitem{Bo23} F.~ Bornemann, \emph{A Stirling-type formula for the distribution of the length of longest increasing subsequences}, Found. Comput. Math. (2023), https://doi.org/10.1007/s10208-023-09604-z, arXiv:2206.09411.

\bibitem{BS09} A.~Borodin and C. D. Sinclair, \emph{The Ginibre ensemble of real random matrices and its scaling limit}, Comm. Math. Phys. \textbf{291} (2009), 177--224.

\bibitem{BL22a} T.~Bothner and A.~Little, \emph{The complex elliptic Ginibre ensemble at weak non-Hermiticity: bulk spacing distributions}, arXiv:2212.00525. 

\bibitem{By23a} S.-S. Byun, \emph{Planar equilibrium measure problem in the quadratic fields with a point charge},  Comput. Methods Funct. Theory (2023). https://doi.org/10.1007/s40315-023-00494-4, arXiv:2301.00324. 

\bibitem{By23b} S.-S. Byun, \emph{Harer-Zagier type recursion formula for the elliptic GinOE}, arXiv:2309.11185.  

\bibitem{BE22} S.-S. Byun and M. Ebke, \emph{Universal scaling limits of the symplectic elliptic Ginibre ensembles}, Random Matrices Theory Appl. \textbf{12} (2023), 2250047.

\bibitem{BES23} S.-S. Byun, M. Ebke and S.-M. Seo, \emph{Wronskian structures of planar symplectic ensembles}, Nonlinearity \textbf{36} (2023), 809--844. 

 \bibitem{BF22} S.-S.~Byun and P. J.~Forrester, \emph{Progress on the study of the Ginibre ensembles I: GinUE}, arXiv:2211.16223.

\bibitem{BF23} S.-S.~Byun and P. J.~Forrester, \emph{Progress on the study of the Ginibre ensembles II: GinOE and GinSE}, arXiv:2301.05022.

\bibitem{BKLL23} S.-S. Byun, N.-G. Kang, J. O. Lee and J. Lee, \emph{Real eigenvalues of elliptic random matrices}, Int. Math. Res. Not. \textbf{2023} (2023), 2243--2280.

\bibitem{BMS23} S.-S.~Byun, L. D.~Molag and N.~Simm, \emph{Large deviations and fluctuations of real eigenvalues of elliptic random matrices}, arXiv:2305.02753.

\bibitem{Ch23}
C. Charlier, \emph{Large gap asymptotics on annuli in the random normal matrix model}, Math. Ann. (2023). https://doi.org/10.1007/s00208-023-02603-z, arXiv:2110.06908.

\bibitem{Ch22}
C. Charlier, \emph{Asymptotics of determinants with a rotation-invariant weight and discontinuities along circles}, Adv. Math. \textbf{408} (2022), 108600.

\bibitem{CES21}
G. Cipolloni, L. Erd\H{o}s and D. Schr\"oder, \emph{Edge universality for non-Hermitian random matrices}, Probab. Theory Related Fields \textbf{179} (2021), 1--28.

\bibitem{CESX22} G. Cipolloni, L. Erd\H{o}s, D. Schr\"oder and Y. Xu, \emph{Directional extremal statistics for Ginibre eigenvalues}, J. Math. Phys. \textbf{63} (2022), 103303.

\bibitem{DDMS19} 
D. S.~Dean, P.~Le Doussal, S. N.~Majumdar and G.~Schehr, \emph{Non-interacting fermions in a trap and random matrix theory}, J. Phys. A \textbf{52} (2019), 144006.

\bibitem{EKS94}
A. Edelman, E. Kostlan and M. Shub, \emph{How many eigenvalues of a random matrix are real?} J. Amer. Math. Soc. \textbf{7} (1994), 247--267.

\bibitem{Efe97}
K. B. Efetov, \emph{Directed quantum chaos}, Phys. Rev. Lett. \textbf{79} (1997), 491.


\bibitem{FS23a} W. FitzGerald and N. Simm, \emph{Fluctuations and correlations for products of real asymmetric random matrices}, Ann. Inst. Henri Poincar\'e Probab. Stat. (to appear) arXiv:2109.00322. 

  
\bibitem{Fo10} P. J. Forrester, \emph{Log-gases and random matrices}, Princeton University Press, Princeton, NJ, 2010.


\bibitem{Fo15a}
P. J. Forrester, \emph{Diffusion processes and the asymptotic bulk gap probability for the real Ginibre ensemble}, J. Phys. A \textbf{48} (2015), 324001.


\bibitem{Fo23}
P. J.~Forrester, \emph{Local central limit theorem for real eigenvalue fluctuations of elliptic GinOE matrices}, arXiv:2305.09124.

\bibitem{FFG06}
P. J.~Forrester, N. E. Frankel and T. M. Garoni, \emph{Asymptotic form of the density profile for Gaussian and Laguerre random matrix ensembles with orthogonal and symplectic symmetry}, J. Math. Phys. \textbf{47} (2006), 023301. 

\bibitem{FLT21} P. J. Forrester, S.-H. Li and A. K. Trinh, \emph{Asymptotic correlations with corrections for the circular Jacobi $\beta$-ensemble}, J. Approx. Th. \textbf{271} (2021), 105633.

\bibitem{FM09}
P. J. Forrester and A. Mays, \emph{A method to calculate correlation functions for $\beta = 1$ random matrices of odd size}, J. Stat. Phys. \textbf{134} (2009), 443--462.

\bibitem{FM12}
P. J. Forrester and A. Mays, \emph{Pfaffian point processes for the Gaussian real generalised eigenvalue problem}, Prob. Theory and Rel. Fields \textbf{154} (2012), 1--47.

\bibitem{FM15} 
P. J. Forrester and A. Mays, \emph{Finite-size corrections in random matrix theory and Odlyzko's dataset for the Riemann zeros}, Proc. A. \textbf{471} (2015), no. 2182, 20150436, 21 pp.

\bibitem{FM23} 
P. J. Forrester and A. Mays, \emph{Finite size corrections relating to distributions of the length of longest increasing subsequences}, Adv. in Appl. Math. \textbf{145} (2023), Paper No. 102482, 33 pp.

\bibitem{FN07} P. J. Forrester and T. Nagao, \emph{Eigenvalue statistics of the real Ginibre ensemble}, Phys. Rev. Lett. \textbf{99} (2007), 050603.

\bibitem{FN08} P. J. Forrester and T. Nagao, \emph{Skew orthogonal polynomials and the partly symmetric real Ginibre ensemble}, J. Phys. A \textbf{41} (2008), 375003.

\bibitem{FT19} P. J. Forrester and A. K. Trinh, \emph{Finite-size corrections at the hard edge for the Laguerre $\beta$ ensemble}, Stud. Appl. Math. \textbf{143} (2019), no. 3, 315--336.

\bibitem{FS23} P. J. Forrester and B.-J. Shen, \emph{Expanding the Fourier transform of the scaled circular Jacobi $\beta$ ensemble density}, arXiv:2306.00525. 


\bibitem{FG23} 
Q. François and D. García-Zelada, \emph{Asymptotic analysis of the characteristic polynomial for the elliptic Ginibre ensemble}, arXiv:2306.16720. 

\bibitem{FK16}
Y. V. Fyodorov and B. A. Khoruzhenko, \emph{Nonlinear analogue of the May-Wigner instability transition}, Proc.  Nat. Acad. Science  \textbf{113} (2016), 6827--6832.

\bibitem{FKS97}
Y. V. Fyodorov, B. A. Khoruzhenko and H.-J. Sommers, \emph{Almost-{H}ermitian random matrices: crossover from Wigner-Dyson to Ginibre eigenvalue statistics}, Phys. Rev. Lett. \textbf{79} (1997), 557--560. 

\bibitem{FKS97a}
Y. V. Fyodorov, B. A. Khoruzhenko and H.-J. Sommers, \emph{Almost-Hermitian random matrices: eigenvalue density in the complex plane}, Phys. Lett. A. \textbf{226} (1997), 46--52.

 
\bibitem{FKS98} 
Y. V. Fyodorov, B. A. Khoruzhenko and H.-J. Sommers, \emph{Universality in the random matrix spectra in the regime of weak non-Hermiticity},  Ann. Inst. H. Poincar\'e Phys. Th\'eor. \textbf{68} (1998), 449--489.

\bibitem{FT21}
Y. V. Fyodorov and  W. Tarnowski, \emph{Condition numbers for real eigenvalues in the real elliptic Gaussian ensemble}, Ann. Henri Poincar\'e \textbf{22} (2021), 309--330.


\bibitem{GFF05}
T. M. Garoni, P. J. Forrester and N. E. Frankel, \emph{Asymptotic corrections to the eigenvalue density of the GUE and LUE}, J. Math. Phys. \textbf{46} (2005), 103301.

\bibitem{GPTZ18}
B. Garrod, M. Poplavskyi, R. Tribe and O. Zaboronski, \emph{Examples of interacting particle systems on Z as Pfaffian point processes: annihilating and coalescing random walks}, Ann. Henri Poincar\'e \textbf{19} (2018), 3635--3662.

\bibitem{GLX23} A. Goel, P. Lopatto and X. Xie, \emph{Central limit theorem for the complex eigenvalues of Gaussian random matrices}, arXiv:2306.10243.

\bibitem{GJ21}
F.~G\"otze, J. Jalowy, \emph{Rate of convergence to the Circular Law via smoothing inequalities for log-potentials}, Random Matrices Theory Appl. \textbf{10} (2021), 25 pp.


\bibitem{GR14} I. S. Gradshteyn and I. M. Ryzhik, \emph{Table of Integrals, Series, and Products}, Elsevier: Academic Press, 2014.


\bibitem{HHJK23} B. C. Hall, C.-W. Ho, J. Jalowy, Z. Kabluchko, \emph{Zeros of random polynomials undergoing the heat flow}, arXiv:2308.11685. 

\bibitem{HW21} H. Hedenmalm and A. Wennman, \emph{Planar orthogonal polynomials and boundary universality in the random normal matrix model}, Acta Math. \textbf{227} (2021), 309--406.


\bibitem{It97} C. Itoi, \emph{Universal wide correlators in non-Gaussian orthogonal, unitary and symplectic random matrix ensembles}, Nucl. Phys. B \textbf{493} (1997), 651--659.

\bibitem{Ja23}
J.~Jalowy, \emph{The Wasserstein distance to the Circular Law}, Ann. Inst. H. Poincar\'e Probab. Stat. (to appear) arXiv:2111.03595. 

\bibitem{JM12} I. M. Johnstone and Z.~Ma, \emph{Fast approach to the Tracy-Widom law at the edge of GOE and GUE}, Ann. Appl. Probab. \textbf{22} (2012), 1962--1988.

\bibitem{KPTTZ15} E.~Kanzieper, M.~Poplavskyi, C.~Timm, R.~Tribe and O.~Zaboronski, \emph{What is the probability that a large random matrix has no real eigenvalues?}, Ann.~Appl.~Probab. \textbf{26} (2016), 2733--2753.

\bibitem{Ku11} A. B. J. Kuijlaars, \emph{Universality}, Chapter 6 in The Oxford Handbook of Random Matrix Theory. In: G. Akemann, J. Baik and P. Di Francesco, (eds.) Oxford University Press (2011).

\bibitem{LR16} S.-Y. Lee and R. Riser, \emph{Fine asymptotic behavior for eigenvalues of random normal matrices: ellipse case}, J. Math. Phys. \textbf{57}, (2016), 023302.


\bibitem{LMS22} A. Little, F. Mezzadri and N. Simm, \emph{On the number of real eigenvalues of a product of truncated orthogonal random matrices}, Electron. J. Probab. \textbf{27} (2022), Paper No. 5, 32 pp.

\bibitem{Mo22}
L. D. Molag, \emph{Edge universality of random normal matrices generalizing to higher dimensions}, Ann. Henri Poincar\'{e} (2023). https://doi.org/10.1007/s00023-023-01333-x, arXiv:2208.12676.

 \bibitem{NO15}
H. H. Nguyen and S. O'Rourke, \emph{The elliptic law}, Int. Math. Res. Not. \textbf{2015} (2015), 7620--7689.

\bibitem{NV21}
O. Nguyen and V. Vu, \emph{Random polynomials: central limit theorems for the real roots}, Duke Math. J. \textbf{170} (2021), 3745--3813.

 \bibitem{OR16}
S. O'Rourke and D. Renfrew, \emph{Central limit theorem for linear eigenvalue statistics of elliptic random matrices}, J. Theoret. Probab.  \textbf{29} (2016), 1121--1191. 

\bibitem{OYZ23}
S. O'Rourke, Z. Yin and P. Zhong, \emph{Spectrum of Laplacian matrices associated with large random elliptic matrices}, arXiv:2308.16171.

\bibitem{NIST}  F. W. J. Olver, D. W. Lozier, R. F. Boisvert, and C. W. Clark, eds. \emph{NIST Handbook of Mathematical Functions}, Cambridge: Cambridge University Press, 2010.

\bibitem{PS16} A. Perret and G. Schehr, \emph{Finite $N$ corrections to the limiting distribution of the smallest eigenvalue of Wishart complex matrices}, Random Matrices Theory Appl. \textbf{5} (2016), 1650001.

\bibitem{SM07}
G.~Schehr and S. N. Majumdar, \emph{Statistics of the number of zero crossings: from random polynomials to the diffusion equation}, Phys. Rev. Lett. \textbf{99} (2007), 060603.

\bibitem{SNWW23}
W. Shi, G. Nemes, X.-S. Wang, and R. Wong, \emph{Error bounds for the asymptotic expansions of the Hermite polynomials}, Proc. Roy. Soc. Edinburgh Sect. A \textbf{153} (2023), no. 2, 417--440.

\bibitem{Si17} N. Simm, \emph{Central limit theorems for the real eigenvalues of large Gaussian random matrices}, Random Matrices Theory Appl. \textbf{6} (2017), 1750002. 

\bibitem{Si17a} 
N. Simm, \emph{On the real spectrum of a product of Gaussian matrices}, Electron. Commun. Probab. \textbf{22} (2017), 11.

\bibitem{Si07}
C. D. Sinclair, \emph{Averages over {G}inibre's ensemble of random real matrices}, Int. Math. Res. Not. \textbf{2007} (2007), rnm015, 15 pp.

\bibitem{Sk59} H. Skovgaard, \emph{Asymptotic forms of Hermite polynomials}, Technical Report 18 (1959) (Department of
Mathematics, California Institute of Technology)


\bibitem{SLMS21} N. R. Smith, P. Le Doussal, S. N. Majumdar and G. Schehr, \emph{Counting statistics for non-interacting fermions in a
$d$-dimensional potential}, Phys. Rev. E \textbf{103} (2021), L030105.

\bibitem{SLMS22}
N. R. Smith, P. Le Doussal, S. N. Majumdar and G. Schehr, \emph{Counting statistics for noninteracting fermions in a rotating trap}, Phys. Rev. A \textbf{105} (2022), 043315.

\bibitem{Sz75} G. Szeg\H o, \emph{Orthogonal Polynomials}, 4th ed. American Mathematical Society, Providence, RI, 1975.

\bibitem{Ta22}
W. Tarnowski, \emph{Real spectra of large real asymmetric random matrices}, Phys. Rev. E \textbf{105} (2022), L012104.

\bibitem{Te96}
N. M. Temme, \emph{Special functions: An introduction to the classical functions of mathematical physics}, John Wiley \& Sons (1996).

\bibitem{TZ11}
R. Tribe and O. Zaboronski, \emph{Pfaffian formulae for one dimensional coalescing and annihilating systems}, Electron. J. Probab. \textbf{16} (2011), 2080--2103.

\bibitem{WCF23} T. R. W\"{u}rfel, M. J. Crumpton and Y. V. Fyodorov, \emph{Mean left-right eigenvector self-overlap in the real Ginibre ensemble}, arXiv:2310.04307. 

\bibitem{YZ23}
L. Yao and L. Zhang, \emph{Asymptotic expansion of the hard-to-soft edge transition}, arXiv:2309.06733.
 
 \end{thebibliography}
\end{document}